\documentclass[11pt]{amsart}
\usepackage[utf8]{inputenc}
\usepackage{fontenc}
\usepackage{amsfonts}
\usepackage{amssymb}
\usepackage{amsmath}
\usepackage{amsthm}\usepackage{mathtools}
\usepackage{enumerate}
\usepackage{hyperref}\usepackage{enumitem}
\usepackage{mathrsfs}
\usepackage{tikz}
\usepackage{marginnote}
\usepackage{xcolor}
\usepackage{soul}

\usepackage[all,cmtip]{xy}

\newcommand{\E}{\mathbb{E}}
\newcommand{\A}{\mathbb{A}}
\newcommand{\B}{\mathbb{B}}
\newcommand{\R}{\mathbb{R}}
\newcommand{\Q}{\mathbb{Q}}

\newcommand{\D}{\mathbb{D}}
\newcommand{\N}{\mathbb{N}}

\newcommand{\cA}{\mathcal{A}}

\newcommand{\cE}{\mathcal{E}}
\newcommand{\cD}{\mathcal{D}}

\newcommand{\eps}{\varepsilon}

\newcommand{\Tr}{\mathrm{Tr}}
\newcommand{\WF}{\mathrm{WF}}
\newcommand{\IF}{\mathrm{IF}}

\newtheorem*{rigprob*}{Rigidity Problem for Uniform Roe Algebras}
\newtheorem*{rigprobcorona*}{Rigidity Problem for Uniform Roe Coronas}

\newcommand{\cC}{\mathcal{C}}

\newcommand{\cF}{\mathcal{F}}
\newcommand{\cP}{\mathcal{P}}

\newcommand{\cB}{\mathcal{B}}
\newcommand{\cK}{\mathcal{K}}

\newcommand{\SD}{\mathrm{SD}}
\newcommand{\SB}{\mathrm{SB}}

\newtheorem{theorem}{Theorem}[section]
\newtheorem*{theorem*}{Theorem}
\newtheorem{proposition}[theorem]{Proposition}

\newtheorem*{proposition*}{Proposition}
\newtheorem{lemma}[theorem]{Lemma}
\newtheorem*{lemma*}{Lemma}
\newtheorem{corollary}[theorem]{Corollary}
\newtheorem*{corollary*}{Corollar}

\newtheorem*{fact*}{Fact}
\theoremstyle{definition}
\newtheorem{definition}[theorem]{Definition}
\newtheorem*{definition*}{Definition}
\newtheorem{claim}[theorem]{Claim}
\newtheorem*{claim*}{Claim}

\newtheorem*{conjecture*}{Conjecture}
\newtheorem{construction}[theorem]{Construction}

\newtheorem{question}[theorem]{Question}
\newtheorem*{acknowledgments}{Acknowledgments}

\theoremstyle{remark}

\newtheorem*{example*}{Example}
\newtheorem{remark}[theorem]{Remark}
\newtheorem*{remark*}{Remark}

\newtheorem*{note*}{Note}
\newtheorem*{question*}{Question}

\DeclareMathOperator{\Span}{span}

\newcommand{\cL}{\mathcal L}

\newcommand{\cW}{\mathcal W}

\newcounter{my_enumerate_counter}
\newcommand{\pushcounter}{\setcounter{my_enumerate_counter}{\value{enumi}}}
\newcommand{\popcounter}{\setcounter{enumi}{\value{my_enumerate_counter}}}

\usepackage{enumitem}

\begin{document}

\title[Uniformly  Factoring Weakly Compact operators]{Uniformly  Factoring Weakly Compact operators and Parametrized Dualization}%

\author{L. Antunes}
\address[L. Antunes] {Departamento de Matem\'atica, Universidade Tecnol\'ogica Federal do Paran\'a, Campus Toledo \\
		85902-490 Toledo, PR \\ Brazil}
\email{leandroantunes@utfpr.edu.br}

\author{K. Beanland}
\address[K. Beanland]{Department of Mathematics, Washington \& Lee University, 204 W. Washington St. Lexington, VA, 24450}
\email{beanlandk@wlu.edu}
\urladdr{kbeanland.wordpress.com}

\author{B. M. Braga}
\address[B. M. Braga]{Department of Mathematics and Statistics,
York University,
4700 Keele Street,
Toronto, Ontario, Canada, M3J
1P3}
\email{demendoncabraga@gmail.com}
\urladdr{https://sites.google.com/site/demendoncabraga}

\subjclass[2010]{}
\keywords{}
\thanks{}
\date{\today}%
\maketitle

\begin{abstract}
This paper deals with the problem of  when, given a  collection $\cC$ of weakly compact operators between separable Banach spaces, there exists a separable reflexive Banach space $Z$ with a Schauder basis so that every element in $\cC$ factors through   $Z$ (or through a subspace of $Z$). A sample result is the existence of a reflexive space $Z$ with a Schauder  basis so that for each separable Banach space $X$, each weakly compact operator from $X$ to $L_1$ factors through $Z$.

We also prove the following descriptive set theoretical result: Let $\cL$ be the standard Borel space of bounded operators between separable Banach spaces. We show that if $\cB$ is   a Borel subset of weakly compact operators between Banach spaces with separable duals, then the assignment $A\in \cB\to A^*$ can be realized by a Borel map $\cB\to \cL$.
\end{abstract}

\setcounter{tocdepth}{1}
\tableofcontents

\section{Introduction}

The celebrated W. Davis, T. Figiel, W.B. Johnson and A. Pe{\l}czy{\'n}ski  factorization theorem   states that every weakly compact bounded operator factors through a reflexive space \cite[Corollary 1]{DavisFiegelJohnsonPelczynski1974} -- recall, a bounded operator $A:X\to Y$ between Banach spaces $X$ and $Y$ is said to \emph{factor through a Banach space $Z$} if      there are bounded operators $B:X \to Z$ and $C:Z \to Y$ so that $A=C\circ B$. This result has been used and generalized in many different directions in the past 45 years. A natural generalization, considered by several authors, asks when a uniform factorization is possible in the following sense: For which collections $\mathcal{C}$ of weakly compact operators between separable Banach spaces does there exist a separable reflexive space $Z$ so that every operator in $\mathcal{C}$ factors through $Z$? In the current paper, we continue this line of inquiry and use several results from descriptive set theory to obtain uniform factorization type theorems for certain collections of weakly compact operators between separable Banach spaces.  

Let us review some of the literature regarding uniform factorization. In 1971, W.B. Johnson  provided a separable space $Z_K$ so that each operator which is the uniform limit of finite rank
operators factors through $Z_K$ \cite[Theorem 1]{Jo-Israel}. On the opposite direction, decades later in 2009, W.B. Johnson and A. Szankowski  showed that there is no separable Banach space through which every compact operator factors    \cite[Theorem 2.5]{JoSz-JFA}. \AA. Lima, O. Nygaard and E. Oja  proved that for each finite dimensional subspace $F$ of weakly compact operators between fixed Banach spaces $X$ and $Y$ there is a reflexive space $Z_F$ so that each operator in $F$ factors (isometrically) through $Z_F$ \cite[Theorem 2.3]{LimaNyOj-Israel}. 

Using the theorem of W. Szlenk  that no separable reflexive space contains isomorphic copies  of all separable reflexive Banach spaces \cite[Theorem 3.3]{Sz}, it follows that there is no separable reflexive space through which every weakly compact operator between separable Banach spaces factors. Therefore uniform factorization questions are only relevant for proper subclasses of weakly compact operators. As it is often common for this kind of results, descriptive set theoretical tools provide us the appropriate framework to study which subclasses have this property (cf. \cite{BeanlandCausey2017,BeanlandFreeman2014,DoFe-Fund}).

There is another natural obstruction for factorization theorems for weakly compact operators. Since a complemented subspace of a Banach space with a Schauder basis must have the bounded approximation property, if  $\mathcal{C}$ is a collection of operators which  contains the identity on a reflexive space without the bounded approximation property and $Z$ is a space so that every operator in $\cC$ factors through, then $Z$ cannot have a basis. It is therefore more tractable to restrict our attention to weakly compact operators which factor through spaces with bases. This, in turn, focuses our attention on weakly compact operators defined on spaces whose domains or codomains have a basis.



In \cite{BeanlandFreeman2014} the second named author and D. Freeman initiated the program of studying uniform factorization problems using descriptive set theory. For that, the authors introduced a natural coding  for the class of all operators between arbitrary Banach spaces, denoted by $\cL$ -- i.e., $\cL$ is a  standard Borel space which naturally codes this class of operators (we refer the reader to Section \ref{SectionBackground} for precise definitions).  This coding was also used in the papers \cite{BeanlandCausey2017,BCFrWa-JFA,BCa-Houston,BCa-Scand,CaNa-JFA}. Since the coding space $\cL$ is a standard Borel space, this allows one to study the complexity of classes of bounded operators between separable Banach spaces, i.e., we can inquire whether certain classes of operators are  Borel, analytic, coanalytic, etc.


The next definition is central to the main results of this paper.
 
 \begin{definition}
Let  $\mathcal{J}\subset \mathcal{L}$ be a coding for some class of operators between separable Banach spaces and $\cP$ be a class of Banach spaces. We say that $\mathcal{J}$ is \emph{strongly bounded with respect to $\cP$} (resp. \emph{complementably strongly bounded with respect to $\cP$}) if for each analytic subset $\mathcal{A}\subset \mathcal{J}$  there exists $Z\in \cP$ so that every  operator in $\mathcal{A}$ factors through a subspace of $Z$ (resp. factors through $Z$).
 \end{definition}
 
The motivation for this terminology comes from the definition of (complementably)  strongly bounded classes of separable Banach spaces. Let us recall this concept. Let $\SB$ be the now classic standard Borel space coding the class of all separable Banach spaces (see Subsection \ref{SubsectionSB} for precise definitions). A subset  $\cB\subset\SB$ is called \emph{strongly bounded} (resp. \emph{complementably strongly bounded}) if for each analytic subset $\cA\subset \cB$ there exists  $Z\in \cB$ so that every element of $\cA$ is isomorphic to a (resp. complemented) subspace of $Z$. Hence, it easily follows that $\cB$ is (resp. complementably) strongly bounded if and only if the collection $\{\mathrm{Id}:X\to X\}_{X\in \cB}\subset \cL$ is (resp. complementably) strongly bounded with respect to $\cB$. A notion of strongly bounded classes of operators which does not depend of a class of Banach spaces was introduced  by the second named author and R. Causey in \cite{BeanlandCausey2017}. Indeed, a collection of operators $\mathcal{J}$ (always between separable Banach spaces) is \emph{strongly bounded} if for each analytic subset $\mathcal{A}$ of $\mathcal{J}$ there is an operator $T \in \mathcal{J}$ so that every operator in $\mathcal{A}$ factors through $T$ restricted to some subspace of its domain. In \cite[Theorem 5.8]{BeanlandCausey2017}, the authors prove that the weakly compact operators are strongly bounded, which, combined with the fact  weakly compact operators factor through reflexive spaces, is equivalent to the statement that the weakly compact operators are strongly bounded with respect to the reflexive spaces. There are many open questions regarding whether a given operator ideal (restricted to separable Banach spaces) is strongly bounded, see \cite{BeanlandCausey2017} for a collection of open problems in this area. 

Furthermore, we note that the weakly compact operators between separable Banach spaces is not complementably strongly bounded with respect to the reflexive spaces. Indeed, as already mentioned above, W.B. Johnson and A. Szankowski \cite{JoSz-JFA} showed there is no separable Banach space through which every compact operator factors and it is routine to observe that the compact operators between separable Banach spaces are a Borel subset of $\mathcal{L}$.

 


 Before stating our main results on complementably strong boundedness, we introduce some notation. Let $\mathcal{W}$ denote the subset  of $\cL$ which codes the weakly compact operators between separable Banach spaces (see Subsection \ref{SubsectionCodeBoundedOp}). 

\begin{definition}\label{DefiWAB}
Given subsets $\mathbb{A},\mathbb{B}\subset \SB$, we define the following collection of operators:
\[\cW_{\mathbb{A},\mathbb{B}}=\{(A:X\to Y)\in \cW: X\in \mathbb{A}\text{ and }Y\in \mathbb{B}\}.\]
If  $\mathbb{A}$ (resp. $\mathbb{B}$) is the isomorphism class $\langle Z\rangle$ in $\SB$ of a certain separable Banach space $Z$, we simply write $\cW_{Z,\mathbb{B}}$ (resp. $\cW_{\mathbb{A},Z}$).
\end{definition}

Let $\SD$,  $\mathrm{sb}$, $\mathrm{sb}^*$,  $\mathrm{ub}$ and $\mathrm{ub}^*$ denote the subsets of $\SB$ coding the separable Banach spaces (1) with separable dual,  (2) with  shrinking bases, (3) whose  duals have  shrinking bases, (4) with  unconditional bases and (5) whose  duals have  unconditional bases, respectively. The following is the main result of \cite{BeanlandFreeman2014}.

\begin{theorem} [Theorem 25 of \cite{BeanlandFreeman2014}]
Both $\mathcal{W}_{\SB,\mathrm{sb}}$ and $\mathcal{W}_{\SB,C(2^\mathbb{N})}$ are complementably strongly bounded with respect to  the class of separable reflexive Banach spaces admitting Schauder bases.
\label{BF-main}
\end{theorem}

We obtain the following result on complementably strongly bounded classes of operators.

\begin{theorem}\label{ThmMainFactResult}
The sets $\mathcal{W}_{\SB,\mathrm{ub}}$,   $\mathcal{W}_{\mathrm{ub}^*,\SD}$, and  $\mathcal{W}_{\mathrm{sb}^*,\SD}$ are complementably strongly bounded with respect to the class of separable reflexive Banach spaces admitting Schauder bases. \label{go through}
\end{theorem}

Moreover, for $\cW_{\SB,L_1}$, we obtain the following stronger result.

\begin{theorem}\label{ThmL1}
There exists a reflexive  $Z\in \SB$ with a Schauder basis so that   every $A\in \cW_{\SB,L_1}$  factors through $Z$.
\end{theorem}

Our main tool to prove Theorem \ref{go through} is a new descriptive set theoretic result which   extends a  theorem by the third named author (see \cite[Theorem 1.1]{Braga2015Fourier}). For that, we will need the concept of isometry between bounded operators between separable Banach spaces, i.e., isometry between elements of $\cL$. Precisely, in \cite[Definition 5.4]{BeanlandCausey2017}, the authors defined the concept of two operators $A:X\to Y$ and $B:Z\to W$ being isomorphic to each other. We define analogously what it means for $A$ and $B$ to be isometric as follows.\footnote{  We notice that, in order to obtain Theorem \ref{ThmMainFactResult}, we could simply work with the concept of isomorphic operators, however, since it is possible to obtain an isometric result in Theorem \ref{ThmLSDmap}, we chose to work with isometric operators.} 

\begin{definition}\label{DefiIsomOperators}
Let $X$, $Y$, $Z$, and $W$ be Banach spaces and $A:X\to Y$ and $B:Z\to W$ be bounded linear operators. We say that $A:X\to Y$ and $B:Z\to W$ are \emph{isometric} (resp. \emph{isomorphic}) if there exist linear isometries (resp. linear isomorphisms) $j_1:X\to Z$ and $j_2:Y\to W$ such that $B\circ j_1=j_2\circ A$, i.e., such that the diagram
\[\xymatrix{X\ar[d]^{j_1}\ar[r]^{A}      &   Y\ar[d]^{j_2}\\
            Z\ar[r]^{B}                    &  W}\]          
commutes. The pair $(j_1,j_2)$ is called an \emph{isometry} (resp. \emph{isomorphism}) between $A$ and $B$. We write $(A:X\to Y)\equiv (B:Z\to W)$ (resp. $(A:X\to Y)\cong (B:Z\to W)$), or  $A\equiv B$ (resp. $A\cong B$) for short. 
\end{definition}

Let $\cL_{\SD}$ denote the subset of $\cL$ which codes all the bounded operators $A:X\to Y$ between Banach spaces $X$ and $Y$ with separable dual. In order to simplify the notation in the technical part of the next theorem, we identify an operator $(A:X\to Y)\in \cL$ with the triple $(X,Y,A)$.\footnote{Starting in Subsection \ref{SubsectionCodeBoundedOp}, we actually identify an operator $A:X\to Y\in\cL$ with a triple $(X,Y,\hat A)$, where $\hat A\in C(\Delta)^\N$. We believe this will cause no confusion and the reason for that will be clear once we present the formal definition of the coding space $\cL$ in Subsection \ref{SubsectionCodeBoundedOp}. }

\begin{theorem}\label{ThmLSDmap}	
Let $\cB\subset \cL_{\SD}$ be a Borel subset. There exists a Borel assignment $A\in \cB\mapsto A^\bullet\in \cL$, i.e., an assignment $(X,Y,  A)\in \cB\mapsto (Y^\bullet,X^\bullet, A^\bullet)\in \cL$, such that $A^\bullet$ is isometric to $A^*$ for all $A\in \cB$.

Moreover, the set 
\[\cC=\{(X,Y,  A,g,x)\in \cB\times C(\Delta)\times C(\Delta):  g\in Y^\bullet, x\in X\}\]
is Borel and there exists a Borel map $[\cdot,\cdot]_{(\cdot,\cdot,\cdot)}:\cC\to \R$ such that, for each $(X,Y,  A)\in \cB$,
\begin{enumerate}
\item $[ \cdot,\cdot]_{(X,Y,   A)}$ is bilinear and norm continuous, and
\item $g\in Y^\bullet\mapsto[ g,\cdot]_{(X,Y,  A)}\in X^*$ is isometric to $A^*$.
\end{enumerate}
\label{dual ops}
\end{theorem}

A corollary of Theorem \ref{dual ops} is the following, which itself is a generalization of the main result of \cite{Do-Houston}.

\begin{corollary}
Let $\mathcal{A} \subset \mathcal{L}_{\SD}$ be analytic. Then the set $$\mathcal{A}^*=\{B \in \mathcal{L}: \exists A \in \mathcal{A} \mbox{ with } B \mbox{ is isomorphic  to }A^*\}$$
is analytic. 
\label{generalize Dodos}
\end{corollary}

  We now describe another corollary of Theorem \ref{ThmLSDmap}. As noticed by G. Godefroy in \cite[Problem 5.2]{GodefroySaintRaymond2018}, the result \cite[Theorem 1.1]{Braga2015Fourier} implies that the equivalence class $\langle X\rangle$ of a separable reflexive Banach space is Borel if and only if the equivalence class of its dual $\langle X^*\rangle $ is Borel.   Theorem \ref{ThmLSDmap} allow us to obtain the analogous result for bounded operators, which generalizes this result for Banach spaces.

Before stating this corollary, we need some terminology. Precisely, given a coding for   bounded operator between separable Banach spaces $\cB\subset \cL$, we define the \emph{isomorphic saturation of $\cB$}, denoted by $\langle\cB\rangle$, by 
\[\langle \cB\rangle=\{(B:Z\to W)\in \cL\colon \exists (A:X\to Y)\in \cB,\ B\cong A\}.\] 
If  $\cB$ is a set containing a single operator, say $\cB=\{A:X\to Y\}$, we simply write $\langle A:X\to Y\rangle$, or $\langle A\rangle$ for short.  If moreover  $\cB\subset \cL_{\SD}$, we define its dual collection $\cB^*$ as in Corollary \ref{generalize Dodos} above.

The following is a consequence of Theorem \ref{ThmLSDmap}.

\begin{corollary}\label{CorOpBetweenReflex}
Let $\cB\subset \cL$ be a collection of operators between reflexive spaces. Then $\langle\cB\rangle$ is Borel if and only if $\cB^*$ is Borel. In particular, if $A:X\to Y$ is an operator between reflexive Banach spaces, then $\langle A\rangle$ is Borel if and only if $\langle A^*\rangle$ is Borel.
\end{corollary}
 
At last, in Section \ref{SectionThmBorelFact} we deal with strongly boundedness over the class of reflexive spaces. As mentioned above, the second named author and R. Causey proved in \cite[Theorem 5.8]{BeanlandCausey2017} that the class of weakly compact operators is strongly bounded over the class of separable reflexive Banach spaces. We are able to obtain the following stronger for the class $\cW_{\SB,\SD}$.

 \begin{theorem}\label{ThmFactoringSubspacesInto}
Let $\mathcal{B}\subset \mathcal{W}_{\SB,\SD}$ be a Borel subset. Then there is a reflexive space $Z \in \SB$ and a $\sigma(\Sigma_1^1)$-measurable $\Psi:\mathcal{B} \to \SB(Z)$ so that each $A=(X,Y,\hat A)\in \cB$ factors through $\Psi(A)$. 
\end{theorem}

Hence, the result above shows not only that given an analytic $\cA\subset\cW_{\SB,\SD}$ one can find a reflexive $Z $ so that each $A\in \cA$ factors through a subspace of $Z$, but moreover that the choice of this subspace can be done is a $\sigma(\Sigma_1^1)$-measurable manner.  We refer the reader to Theorem \ref{ThmFactoringSubspacesInto} below for a stronger more technical result.

We finish the paper presenting some open problems in Section \ref{SectionOpenProb}.

\section{Preliminaries}\label{SectionBackground}

The Banach space theory terminology  used herein is standard and we refer the reader to \cite{AlbK-book}. We emphasize here that all Banach spaces are considered to be over the reals and, given a Banach space $X$, we denote its closed unit ball by $B_X$. We also write $X\cong Y$ and $X \equiv Y$ to denote that $X$ and $Y$ are isomorphic and isometric, respectively. For the background on descriptive set theory, we refer the reader to \cite{Ke-book}.

\subsection{Polish spaces and coding  separable Banach spaces}\label{SubsectionSB}

A separable topological space $(X,\tau)$, where $\tau$ is a topology on the set $X$, is called a \emph{Polish space} if there exists a complete metric on $X$ which generates the topology $\tau$. In this case, we say that $\tau$ is a \emph{Polish topology}.

A measurable space $(X,\mathcal{A})$, where $\mathcal{A}$ is a $\sigma$-algebra on the set $X$, is called a \emph{standard Borel space} if there exists a Polish topology on $X$ so that $\cA$ is the Borel $\sigma$-algebra generated by this topology.

Since the class of all separable Banach spaces is a proper class, i.e., it is not a set, in order to study the descriptive set theoretical properties of the class of all separable Banach spaces it is necessary to first code this class as a set. This is usually done as follows. Denote the Cantor set by $\Delta$. Since the Banach space of continuous real-valued functions on the Cantor set, $C(\Delta)$, is isometrically universal for all separable Banach spaces\footnote{i.e., every separable Banach space embeds into $C(\Delta)$ isometrically \cite[Page 79]{Ke-book}.}, we define 
\[\SB=\{X\subset C(\Delta): X\text{ is a closed linear subspace}\}\]
and endow $\SB$ with the Effros-Borel $\sigma$-algebra, i.e., the $\sigma$-algebra generated by the sets
 \[\Big\{\{F\in \SB: U\cap F\neq \emptyset\}: U\subset C(\Delta) \text{ is open}\Big\}.\]
It is well-known that $\SB$  endowed with the Effros-Borel structure is a standard Borel space \cite[Theorem 2.2]{Do-Book}. Similarly, if $Z\in \SB$, we define 
\[\SB(Z)=\{X\in \SB: X\subset Z\}.\]
So $\SB(Z)$ is a Borel subset of $\SB$, hence a standard Borel space.

With this coding in hand, it makes sense to ask whether specific classes of separable Banach spaces are Borel, analytic, coanalytic, etc. Moreover, any other ``reasonable coding" for the separable Banach spaces agrees with this coding (see \cite[Proposition  2.8]{Bos-FunD}).

\subsection{Coding  bounded operators}\label{SubsectionCodeBoundedOp}
Let us introduce    the coding for operators between separable Banach spaces defined in \cite[Section 4]{BeanlandFreeman2014} (cf. \cite[Section 8.2]{BCFrWa-JFA}). Throughout this paper, let  \[\Big(d_n:\SB\to C(\Delta)\Big)_n\] be a sequence of Borel maps so that  (1) $\{d_n(X)\}_n$ is a dense subset of   $X$ for all $X\in \SB$, and (2) for all $m,k\in \N$ and all $r,q\in \Q$ there exists $n\in\N$ such that $d_n(X)=rd_k(X)+qd_m(X)$ -- the existence of such sequence is given by   Kuratowski
and Ryll-Nardzewski's selection theorem \cite[Theorem 12.13]{Ke-book}. 

We use the notation  $\Q_{>0}=\Q\cap (0,\infty)$ throughout.

\begin{definition}
We define the \emph{coding for all bounded operators between separable Banach spaces} as the subset $\cL\subset \SB\times \SB \times C(\Delta)^\N$ given by
\begin{align*}
(X,Y,\hat A)&\in  \cL\ \\
\Leftrightarrow\  &\Big(\forall n\in\N, \ \hat A(n)\in Y\Big)\wedge\Big(\forall n,m,k\in\N,\ \forall r,q\in \Q,\\
&  d_n(X)=rd_k(X)+qd_m(Y)\Rightarrow \hat A(n)=r\hat A(k)+q\hat A(m)\Big)\wedge\\
&\Big(
\exists L\in \Q_{>0},\ \forall (a_i)_i\in \Q^{<\N},\ \Big\|\sum_ia_i\hat A(i)\Big\|\leq L\Big\|\sum_ia_id_i(X)\Big\|\Big).
\end{align*} 
\end{definition}

Firstly, notice that  $\cL$ is clearly a Borel subset of $\SB\times \SB \times C(\Delta)^\N$  (see \cite[Section 8.2]{BCFrWa-JFA} for details), so $\cL$ is a standard Borel space. We now explain how $\cL$ codes  the  operators between separable Banach spaces. Given Banach spaces $X,Y\in \SB$ and a bounded operator $A:X\to Y$, we associate to $A:X\to Y$ the triple $(X,Y,\hat A)$, where $\hat A=(A(d_n(X)))_n$. Clearly $(X,Y,\hat A)\in \cL$. On the other hand, if $(X,Y,\hat A)\in \cL$, define $A:X\to Y$ to be the unique operator such that $A(d_n(X))=\hat A(n)$ for all $n\in\N$. For the remaining of this paper, using the identification just described, we make no distinction between the triple $(X,Y,\hat A)\in \cL$ and the operator $A:X\to Y$ (cf. \cite[Claim 8.4]{BCFrWa-JFA}). Moreover, if the spaces $X$ and $Y$ can be neglected, we simply write $A\in \cL$.

Let $\SD=\{X\in \SB : X^*\text{ is separable}\}$ and define
\[\cL_{\SD}=\{(X,Y,\hat A)\in \cL: X,Y\in \SD\}.\]
So, $\cL_{\SD}=\cL\cap (\SD\times\SD\times C(\Delta)^\N)$. Since $\SD$ is coanalytic  \cite[Corollary 3.3(ii)]{Bos-FunD}, so is $\cL_{\SD}$. Moreover, let $\Tr$ be the standard Borel space of trees on $\N$ (see Subsection \ref{SubsectionLeftmost} for  definitions regarding  trees). Since $\SD$ is complete coanalytic (see the proof of \cite[Corollary 3.3(ii)]{Bos-FunD}), there exists a Borel reduction $\varphi:\Tr\to \SB$  of the well-founded trees, $\WF$, to $\SD$, i.e., $\varphi$ is a Borel map so that $\varphi^{-1}(\SD)=\WF$. Define $\Phi:\Tr\to \cL$ by letting
\[\Phi(T)=(\varphi(T),\varphi(T), (d_n(\varphi(T)))_n)\]
for all $T\in \Tr$. Then $\Phi$ is a Borel reduction of $\WF$ to $\cL_{\SD}$. So, $\cL_{\SD}$ is also complete coanalytic. In particular, $\cL_\SD$ is not Borel.  

We denote the coding for the weakly compact operators by $\cW$, i.e., \[\cW=\{ A\in \cL : A\text{ is weakly compact}\}.\] By \cite[Proposition 22]{BeanlandFreeman2014}, $\cW$ is coanalytic. At last, 
we denote the coding for the weakly compact operators between spaces with separable dual by $\cW_{\SD}$, i.e., \[\cW_{\SD}=\cW\cap\cL_\SD.\]
Alternatively, using the terminology of Definition \ref{DefiWAB}, we have $\cW_{\SD}=\cW_{\SD,\SD}$.

\subsection{Hyperspace}
Denote the set  of all compact subsets of a compact metric space $S$ by $\cK(S)$, and endow $\cK(S)$ with the Vietoris topology. Since $S$ is a metric space, the Vietoris topology is generated by the Hausdorff metric, and this metric makes $\cK(S)$ into a compact Polish space. The reader can find more details about $\cK(S)$ and its topology in \cite[Section 4.F]{Ke-book}.

The next  lemma is \cite[Lemma 3.7]{Braga2015Fourier} and it will play an important role in Section \ref{SectionThmBorelFact}.   Recall that a critical ingredient towards showing that a separable Banach space $X$ embeds isometrically into $C(\Delta)$ is the fact that $B_{X^*}$ is separable and metrizable in the weak$^*$ topology and thus image of the Cantor set under some continuous map. The next lemma gives us a way to parametrize the selection of this continuous surjection and it is our tool to bring abstract Banach spaces into a concrete  isometric copy of it in $C(\Delta)$.

\begin{lemma} \label{LemmaH} 
Let $\Delta$ be the Cantor set. There exists a Borel function 
\[Q:\cK(\Delta)\to C(\Delta, \Delta)\]
such that, for each $K\in \cK(\Delta)$, $Q(K):\Delta\to \Delta$ is a continuous function onto $K$. Therefore, if $M$ is a compact metric space, and $h:\Delta \to M$ is a continuous surjection, we have that
\[H:K\in\cK(M)\mapsto  h\circ Q(h^{-1}(K))\in C(\Delta, M),\]
is a Borel function and, for each $K\in \cK(M)$, $H(K):\Delta\to M$ is a continuous function onto $K$.
\end{lemma}

\subsection{Trees}\label{SubsectionLeftmost}

For a set $\Lambda$, we let $\Lambda^{<\N}=\bigcup_{n\in\N}\Lambda^n$. If $\bar n=(n_1,\ldots, n_k)\in \Lambda^{<\N}$, we write $|\bar n|=k$, and $|\bar n|$ is called the \emph{length} of $\bar n$. If $\ell\leq k$, we write $\bar n\restriction \ell=(n_1,\ldots, n_\ell)$. Analogously, if $\bar n=(n_j)_{j=1}^\infty\in \Lambda^\N$ and $\ell\in\N$, we write $\bar n\restriction \ell =(n_1,\ldots,n_\ell)$. Define an order $\preceq $ on $\Lambda^{<\N}$ by setting $\bar n\preceq \bar m$ if $\bar n$ is an initial segment of $\bar m$, i.e., if $|\bar n|\leq |\bar m|$ and  $\bar m\restriction |\bar n|=\bar n$. If $\bar m\in \Lambda^\N$, we define $\bar n\preceq \bar m$ analogously. A subset $T\subset \Lambda^{<\N}$ is called a \emph{tree on $\Lambda$} if for all $t\in T$ and all $s\in \Lambda^{<\N}$,   $s\preceq t$ implies $s\in T$. A tree $T$ is called \emph{pruned} if for all $s\in T$ there exists $t\in T\setminus \{s\}$ with $s\preceq t$. 

If $T$ is a tree on a set $\Lambda$, an element $\beta\in \Lambda^\N$ is called a \emph{branch of $T$} if $\beta\restriction {k}\in T$ for all $k\in\N$, and we denote the subset of $\Lambda^\N$ containing all branches of $T$ by $[T]$. We say that a tree $T$ on $\Lambda$ is \emph{ill-founded} if $[T]\neq \emptyset$, otherwise $T$ is called \emph{well-founded}. If $\Lambda =\N$, we denote the set all all trees on $\N$ by $\Tr$, and the subsets of all ill-founded trees and well-founded trees by $\IF$ and $\WF$, respectively.

\section{The adjoint map as a Borel function}\label{SectionDUAL}
 
In this section, we prove Theorem \ref{ThmLSDmap}. As mentioned in the introduction, Theorem \ref{ThmLSDmap} is a generalization of \cite[Theorem 1.1]{Braga2015Fourier}. To prepare for the proof, we first   introduce  a coding for the unit ball of the duals of separable Banach spaces as compact subsets of the product space $[-1,1]^\N$  (we follow the approach of \cite[Section 2.1.2]{Do-Book}). 

 Given $X\in \SB$, define $K_{X^*}\subset [-1,1]^\N $ by\footnote{As a rule, we denote elements of a space $X$ by $x$, functionals in the dual $X^* $ by $f$,  elements in the product  space $[-1,1]^\N$ by $x^*$, and their $n$-th coordinate by $x^*_n$.}
\[K_{X^*}=\Big\{x^*\in [-1,1]^\N:  \exists f\in B_{X^*},  \ \forall n\in\N,\   x^*_n=\frac{f(d_n(X))}{\|d_n(X)\|}\Big\},\]
where if $d_n(X)=0$, we let $x^*_n=0$ above. It is easy to see that the subset $\D\subset \SB\times [-1,1]^\N$ given by 
\[(X,x^*)\in \D\Leftrightarrow x^*\in K_{X^*}\]
is Borel and that $K_{X^*}$ is a compact  subset of $[-1,1]^\N$ for all $X\in \SB$ (see \cite[Section 2.1.2]{Do-Book}). Therefore, by \cite[Theorem 28.8]{Ke-book}, the map 
\[X\in \SB\mapsto K_{X^*}\in \cK([-1,1]^\N)\]
is Borel. Moreover,  one can easily see that for each $X\in \SB$ there exists a linear isometry \[i_X:\overline{\Span}\{K_{X^*}\}\to {X^*}\] such that 
\begin{equation}\label{EqPropIsometryiX}
 f=i_{X}(x^*)\text{ implies } x^*_n=
 \left\{\begin{array}{ll}
 \frac{f(d_n(X))}{\|d_n(X)\|},&\text{ if }d_n(X)\neq 0,\\
 0,&\text{ if }d_n(X)=0,\end{array}\right. 
\text{ for all }n\in\N\tag{$*$}
\end{equation}  
 (see \cite[Section 2.1.2]{Do-Book} or \cite[Section 3]{Braga2015Fourier}).
 
We will also need a Borel way of computing the functional evaluation of an element $x^*\in K_{X^*}$ at some $x\in X $. We proceed as follows. Consider the Borel set
\[\A=\{(X,x,x^*)\in \SB\times C(\Delta)\times [-1,1]^\N:  x\in X,x^*\in K_{X^*}\},\]
and define a map $\alpha:\A\to \R$ by letting 
\[\alpha(X,x,x^*)=\langle i_X(x^*),x\rangle\]
for all $(X,x,x^*)\in \A$, where, given $f\in X^*$ and $x\in X$, $\langle f,x\rangle$ denotes the functional evaluation of $f$ at $x$, i.e., $f(x)$. By \cite[Lemma 3.8]{Braga2015Fourier}, $\alpha$ is a  Borel map. 

We are now in good shape to present a proof for Theorem \ref{ThmLSDmap}.

\begin{proof}[Proof of Theorem \ref{dual ops}]
Since $\cB$ is Borel, the subsets
\[\A_1=\{X\in \SB:  \exists Y\in \SB,\ \exists \hat A\in C(\Delta)^\N\text{ s.t. }(X,Y,\hat A)\in \cB\}\]
and 
\[\A_2=\{Y\in \SB:  \exists X\in \SB,\ \exists \hat A\in C(\Delta)^\N\text{ s.t. }(X,Y,\hat A)\in \cB\}\]
are analytic. Since $\A_1\cup\A_2\subset \SD$ and $\SD$ is coanalytic,   Lusin's separation theorem \cite[Theorem 14.7]{Ke-book} gives  a Borel subset $\B\subset \SD$ such that $\A_1\cup\A_2\subset \B$. 
 Define
\[\E=\{(X,x^*,y)\in \B\times [-1,1]^\N\times \Delta:  x^*\in K_{X^*}\} \]
and let   $\gamma:\E\to \R$ be the Borel map given by  \cite[Corollary 3.10]{Braga2015Fourier}.  The properties of $\gamma$ which are important for our goals are 
\begin{enumerate}
\item $\gamma(X,x^*,\cdot)\in C(\Delta)$ for all $X\in \B$ and all $x^*\in K_{X^*}$, and 
\item $\|\gamma(X,x^*,\cdot)\|=\|x^*\|_\infty$ for all $X\in \B$ and all $x^*\in K_{X^*}$.
\end{enumerate} Given $X\in \B$ and $x^*\in K_{X^*}$, the reader should interpret $\gamma(X,x^*,\cdot)$ as the coding -- back in the universal space $C(\Delta)$ -- of the functional of $X^*$ which is coded by $x^*\in K_{X^*}$.

For each $X\in \B$, define 
\[X^\bullet=\Big\{g\in C(\Delta):  \exists x^*\in K_{X^*},\ \exists \lambda \in \R, \text{ s.t. } g(y)=\lambda\gamma(X,x^*,y),\  \forall y\in \Delta\Big\}.\]
It is shown in the proof of \cite[Theorem 1.1]{Braga2015Fourier} that the assignment $X\in \B\mapsto X^\bullet\in \SB$ is Borel and that $X^\bullet$ is linearly isometric to  $X^*$ for all $X\in \B$. Moreover, let \[\cA=\{(X,x,g)\in \B\times C(\Delta)\times C(\Delta):  x\in X, g\in X^\bullet\}\] and for each $(X,x,g)\in \cA$, with $g=\lambda\gamma(X,x^*,\cdot)$, define 
\[\langle g,x\rangle_X=\lambda\alpha(X,x,x^*).\]
Then $\cA$ is Borel and it was proved in the proof of \cite[Theorem 1.1]{Braga2015Fourier}  that $\langle g,x\rangle_X$ does not depend on the representative $\lambda\gamma(X,x^*,\cdot)$ of $g$. At last, the proof of \cite[Theorem 1.1]{Braga2015Fourier} also shows that \[(X,x,g)\in \cA\mapsto  \langle g,x\rangle_X\in \R\] is Borel and that, for each $X\in \B$, 
\begin{enumerate}\setcounter{enumi}{2}
\item $\langle\cdot,\cdot\rangle_X$ is bilinear and norm continuous, and 
\item $g\in X^\bullet\mapsto \langle g,\cdot\rangle_X\in X^*$ is a surjective linear isometry. 
\end{enumerate}

 We now return to the proof of the theorem.  For that, we  define an assignment $(X,Y,\hat A)\in \cB\mapsto \hat A^\bullet\in C(\Delta)^\N$ so that the sequence $\hat A^\bullet\in C(\Delta)^\N$ will work as the coding of the adjoint operator $A^*:Y^*\to X^*$ analogously as $\hat A$ codes the operator $A:X\to Y$. For each $n\in\N$ define a map $x^*_n:\cB\to \R^\N$ by letting 
\[x^*_n(X,Y,\hat A)=\Big(\Big\langle  d_n(Y^\bullet),\frac{\hat A(m)}{\|d_m(X)\|}\Big\rangle_Y\Big)_m\in  \R^\N,\]
where $ \hat A(m)/\|d_m(X)\|$ is taken to be $0$ above if $d_m(X)=0$. Since the maps $(X,Y,\hat A)\in \cB\mapsto \hat A(m)\in C(\Delta)$, $(X,x,g)\in \cA\mapsto \langle g,x\rangle_X\in \R$, and $X\in \SB\mapsto d_m(X)$ are Borel for all $m\in\N$,  the  map $x^*_n$ is a Borel map for all $n\in\N$. 

For each $n\in\N$, define $\hat A^\bullet(n)\in C(\Delta)$ by letting $\hat A^\bullet(n)=0$ if $x^*_n(X,Y,\hat A)=0$, and 
\begin{equation}\label{EqFormulaHatABullet}
\hat A^\bullet(n)=\|x^*_n(X,Y,\hat A)\|\cdot \gamma\Big(X,\frac{x^*_n(X,Y,\hat A)}{\|x^*_n(X,Y,\hat A)\|},\cdot\Big)\tag{$**$}
\end{equation}
otherwise. By the definition of $X^\bullet$, $\hat A^\bullet(n)\in X^\bullet$ for all $n\in\N$.

 Clearly, $(X,Y,\hat A,y)\in \cB\times \Delta\mapsto \hat A^\bullet(n)(y)\in \R$ is Borel, which implies that   $(X,Y,\hat A)\in \cB\mapsto \hat A^\bullet(n)\in C(\Delta)$ is Borel for all $n\in\N$. So,  $(X,Y,\hat A)\in \cB\mapsto \hat A^\bullet\in C(\Delta)^\N$ is Borel. In particular, the assignment  \[(X,Y,\hat A)\in \cB\mapsto (Y^\bullet,X^\bullet, \hat A^\bullet)\in \SB\times \SB\times  C(\Delta)^\N\] is Borel.

\begin{claim}
For each $(X,Y,\hat A)\in \cB$, we have that $(Y^\bullet,X^\bullet, \hat A^\bullet)\in \cL$. In particular,   the assignment $(X,Y,\hat A)\in \cB\mapsto (Y^\bullet,X^\bullet, \hat A^\bullet)\in \cL$ is well-defined (and it is Borel). 
\end{claim}

\begin{proof}
Fix $(X,Y,\hat A)\in \cB$. Let $q_1,\ldots,q_k\in \Q$ and $n_1,\ldots, n_k\in\N$. Since $g\in X^\bullet\mapsto \langle g,\cdot\rangle_X\in X^*$ is a surjective linear isometry and $\langle\cdot,\cdot\rangle_X$ is bilinear, we have that 
\begin{align*}
\Big\|\sum_{i=1}^kq_i\hat A^\bullet(n_i)\Big\|&=\sup_{m\in\N}\Big|\Big\langle\sum_{i=1}^kq_i\hat A^\bullet(n_i), \frac{d_m(X)}{\|d_m(X)\|}\Big\rangle_X\Big|\\
&=\sup_{m\in\N}\Big|\sum_{i=1}^k q_i\Big\langle \hat A^\bullet(n_i),  \frac{d_m(X)}{\|d_m(X)\|}\Big\rangle_X\Big|.
\end{align*}
By the definition of $\langle\cdot,\cdot\rangle_X$ and using the representative of $\hat A^\bullet(n_i)$ given in \eqref{EqFormulaHatABullet}, this gives that
\begin{align*}
\Big\|\sum_{i=1}^k &q_i\hat A^\bullet(n_i)\Big\|\\
&=\sup_{m\in\N}\Big|\sum_{i=1}^k q_i\|x^*_{n_i}(X,Y, \hat A)\|\cdot \alpha\Big(X, \frac{d_m(X)}{\|d_m(X)\|},\frac{x^*_{n_i}(X,Y,\hat A)}{\|x^*_{n_i}(X,Y,\hat A)\|}\Big)\Big|.
\end{align*}
At last, since $\alpha(X,x,x^*)=\langle i_{X}(x^*),x\rangle$, the property of $i_X$ given in \eqref{EqPropIsometryiX} implies that
\begin{align*}
\Big\|\sum_{i=1}^kq_i\hat A^\bullet(n_i)\Big\|&=\sup_{m\in\N}\Big|\sum_{i=1}^k q_i\Big\langle  d_{n_i}(Y^\bullet),  \frac{\hat A(m)}{\|d_m(X)\|}\Big\rangle\Big|\\
&=\sup_{m\in\N}\Big|\Big\langle\sum_{i=1}^k q_i d_{n_i}(Y^\bullet),  \frac{\hat A(m)}{\|d_m(X)\|}\Big\rangle\Big|\\
&\leq \sup_{m\in\N}\Big\|\frac{\hat A(m)}{\|d_m(X)\|}\Big\|\cdot\Big\|\sum_{i=1}^k q_i d_{n_i}(Y^\bullet)\Big\|  \\
&=\|A\|\cdot \Big\|\sum_{i=1}^k q_i d_{n_i}(Y^\bullet)\Big\|.
\end{align*}
Since $q_1,\ldots,q_k\in \Q$ and $n_1,\ldots, n_k\in\N$ are arbitrary, it follows that   $(Y^\bullet, X^\bullet, \hat A^\bullet)\in \cL$, i.e., $\hat A^\bullet $ codes an operator $A^\bullet:Y^\bullet\to X^\bullet$.
\end{proof}

\begin{claim}\label{ClaimABulletIsoA*}
The operator $A^\bullet:Y^\bullet\to X^\bullet$ is isometric to $A^*:Y^*\to X^*$ for all $(X,Y,\hat A)\in \cB$.
\end{claim}

\begin{proof} For each $Z\in \{X,Y\}$, define $j_Z:Z^\bullet\to Z^*$ by letting $j_Z(g)=\langle g,\cdot\rangle_Z\in Z^*$ for all $g\in Z^\bullet$. So, both $j_X$ and $j_Y$ are linear isometries.  Let us show that  $A^*\circ j_Y=j_X\circ A^\bullet$.

Fix $n\in\N$.  If $d_n(Y^\bullet)=0$, then clearly $A^*\circ j_Y(d_n(Y^\bullet))=j_X\circ A^\bullet(d_n(Y^\bullet))$. Assume $d_n(Y^\bullet)\neq 0$. Given $m\in \N$, we have that 
\[A^*\circ j_Y(d_n(Y^\bullet))(d_m(X)) =j_Y(d_n(Y^\bullet))(A(d_m(X)))
=\langle d_n(Y^\bullet),\hat A(m)\rangle_Y.
\]
By the definition of $x^*_n(X,Y,\hat A)$ and the property of $i_X$ given in \eqref{EqPropIsometryiX}, it follows that
\[\langle i_X(x^*_n(X,Y,\hat A)), d_m(X)\rangle=\langle d_n(Y^\bullet),\hat A(m)\rangle_Y.\]
By the definition of $\alpha$ and $\gamma$, we conclude that  
\begin{align*}
\langle i_X(x^*_n(X,Y,& \hat A)), d_m(X)\rangle
\\ 
&=\|x^*_n(X,Y,\hat A)\|\cdot \alpha\Big(X,d_m(X),\frac{x^*_n(X,Y,\hat A)}{\|x^*_n(X,Y,\hat A)\|}\Big)\\
&=j_X\Big(\|x^*_n(X,Y,\hat A)\|\cdot \gamma\Big(X,\frac{x^*_n(X,Y,\hat A)}{\|x^*_n(X,Y,\hat A)\|},\cdot\Big)\Big)(d_m(X))\\
&=\Big(j_X\circ A^\bullet(d_n(Y^\bullet))\Big)(d_m(X)).
\end{align*}
Since $(d_m(X))_m$ is dense in $X$, it follows that \[A^*\circ j_Y(d_n(Y^\bullet))=j_X\circ A^\bullet(d_n(Y^\bullet))\] for all $n\in\N$, and we conclude that $A^*\circ j_Y=j_X\circ A^\bullet$.
\end{proof}

This concludes the proof of the first part of the theorem. We now move to its second statement.

\begin{claim}\label{ClaimCompNormOp}
The assignment  $(X,Y,\hat A)\in \cL\mapsto \|A\|\in \R$ in Borel. 
\end{claim}

\begin{proof}
This follows since for each $n\in\N$ the assignment 
\[(X,Y,\hat A)\in \cL\mapsto \frac{\hat A(n)}{\|d_n(X)\|}\in \R,\] where $\hat A(n)/\|d_n(X)\|$ is taken to be $0$ if $d_n(X)=0$, is Borel. Hence, since  $\|A\|=\sup_n\hat A(n)/\|d_n(X)\|$,  it follows that $(X,Y,\hat A)\in \cL\mapsto \|A\|\in \R$ in Borel.
\end{proof}

\begin{claim}\label{ClaimMapABulletgBorel}
Let 
\[\cD=\{(X,Y,\hat A,g)\in \cB\times C(\Delta):  g\in Y^\bullet\}.\]
Then $\cD$ is Borel and the map  
\[(X,Y,\hat A,g)\in \cD \mapsto A^\bullet(g)\in C(\Delta)\]
is Borel.
\end{claim}

\begin{proof}
Recall the definition of the operator $A^\bullet:Y^\bullet\to X^\bullet$ coded by $\hat A^\bullet$. Given $g\in Y^\bullet$ and a sequence $(n_i)_i$ of natural numbers such that $g=\lim_id_{n_i}(Y^\bullet)$, we have $A^\bullet(g)=\lim_i\hat A^\bullet(n_i)$. Moreover, by the definition of $\cL$ and since $\|A^\bullet\|=\|A\|$, we have that
\[\|A^\bullet (g)-\hat A^\bullet(n)\|\leq \|A\|\cdot\|g-d_n(Y^\bullet)\|\]
for all $n\in\N$. 

Let $h\in C(\Delta)$ and $\eps>0$. By the above, it follows that 
\begin{align}\label{EqIFF}
\|A^\bullet (g)-h\|\leq \eps\Leftrightarrow\ &\forall  \delta\in (0,\eps)\cap \Q,\ \exists n\in\N \text{ s.t.}\\
&\Big(\|g-d_n(Y^\bullet)\|\leq \frac{\delta}{\|A\|}\Big)\wedge\Big(\|\hat A^\bullet(n)-h\|\leq\eps+\delta\Big).\nonumber
\end{align}
By Claim \ref{ClaimCompNormOp}, the right hand side of \eqref{EqIFF} gives us a Borel condition on $(X,Y,\hat A,g)\in \cD$. So, the claim follows.
\end{proof}

Let 
\[\cC=\{(X,Y,\hat A,g,x)\in \cB\times C(\Delta)\times C(\Delta):  g\in Y^\bullet, x\in X\}.\]
Since $Y\in \cB\mapsto Y^\bullet$ is Borel, $\cC$ is Borel. Given $(X,Y,\hat A,g,x)\in \cC$, define
\[[g,x]_{(X,Y,\hat A)}=\langle A^\bullet (g),x\rangle_{X}.\]
By Claim \ref{ClaimMapABulletgBorel}, $[\cdot,\cdot]_{(\cdot,\cdot,\cdot)}:\cC\to \R$ is Borel. 

Fix $(X,Y,\hat A)\in \cB$. It is  clear that  $[\cdot,\cdot]_{(X,Y,\hat A)}$ is bilinear and norm continuous. At last, since $[d_n(Y^\bullet),d_m(X)]_{(X,Y,\hat A)}=\langle A^\bullet (d_n(Y^\bullet)),d_m(X)\rangle_{X}$,  Claim \ref{ClaimABulletIsoA*} implies that 
\[A^*\circ j_Y(d_n(Y^\bullet))(d_m(X))=[d_n(Y^\bullet),d_m(X)]_{(X,Y,\hat A)}\]
for all $n,m\in\N$. So, $(j_Y,\text{Id}_{X^*})$ is an isometry between $g\in Y^\bullet\mapsto[g,\cdot]_{(X,Y,\hat A)}\in X^*$ and  $A^*:Y^*\to X^*$. This finishes the proof of the theorem.
\end{proof}

We conclude this section with the proofs of Corollaries \ref{generalize Dodos} and \ref{CorOpBetweenReflex}.
But before that let us point out two straightforward facts about any collection $\cB\subset \cL$. Precisely, given such $\cB$, we have that 
  \begin{enumerate}
      \item if $\cB$ is analytic, so is $\langle\cB\rangle$ \cite[Lemma 5.17]{BeanlandCausey2017}, and
      \item if $\cB\subset \cL_{\SD}$, then $\cB^*$ is  closed under isomorphism, i.e., $\cB^*=\langle \cB^*\rangle$.
  \end{enumerate}

\begin{proof}[Proof of Corollary \ref{generalize Dodos}]
 Let  $\mathcal{A} \subset \cL_{\SD}$ be analytic. As $\cL_{\SD}$ is coanalytic, Lusin's separation theorem implies there is a Borel subset $\mathcal{B}$ of $\cL_{\SD}$ with $\mathcal{A} \subset \mathcal{B}$. Theorem \ref{dual ops} yields a Borel assignment, call it $D:\cB \to \mathcal{L}$, with $D(X,Y,  A)=(Y^\bullet,X^\bullet, A^\bullet)$, such that $A^\bullet$ is isometric to $A^*$ for each $A\in \cB$. As analyticity if preserved under Borel maps, $D(\cA)$ is analytic. Likewise, the isomorphic saturation of an analytic set is analytic   \cite[Lemma 5.17]{BeanlandCausey2017} and thus the following set is analytic:
$$ \{ B \in \mathcal{L} : \exists C \in D(\cA) \mbox{ with $B$ isomorphic  to $C$}\}.$$
Since the above set is equal to the  $\mathcal{A}^*$ in the statement of the  corollary, this concludes the proof. \end{proof}

\begin{proof}[Proof of Corollary \ref{CorOpBetweenReflex}]
For the forward direction, suppose $\langle\cB\rangle$ is Borel. Let us first notice that $\cB^*$ is analytic. For that, consider the Borel assignment  $(X,Y,\hat A)\in \langle\cB\rangle\mapsto (Y^\bullet,X^\bullet, \hat A)\in \cL$ given by Theorem \ref{ThmLSDmap}. Then
\[\cB^*=\big\langle \{(Y^\bullet,X^\bullet,\hat A^\bullet)\in \cL\colon  (X,Y,\hat A)\in \langle\cB\rangle\}\big\rangle.\]
Since the assignment above is Borel, $\{(Y^\bullet,X^\bullet,\hat A^\bullet)\in \cL\colon (X,Y,\hat A)\in \langle\cB\rangle\}$ is analytic, which implies that  $\cB^*$ is  analytic. 

Since $\langle\cB\rangle$ is a collection of operators between reflexive spaces, we have that $\cB^*\subset \cL_{\SD}$. Since $\cL_{\SD}$ is coanalytic,   Lusin's separation theorem allow us to pick a Borel $\cC\subset \cL_{\SD}$ so that $\cB^*\subset \cC$. For now on, let $\Phi:(X,Y,\hat A)\in \cC\mapsto (Y^\bullet,X^\bullet, \hat A)\in \cL$ be the Borel assignment given by Theorem \ref{ThmLSDmap}. Since $\langle\cB\rangle$ is a collection of operators between reflexive spaces it follows that   $\cB^*=\Phi^{-1}(\langle\cB\rangle)$. Hence, $\cB^*$ is Borel since $\langle\cB\rangle$ is Borel.

For the backwards direction, suppose now that $\cB^*$ is Borel. By the definition of $\cB^*$, we have that $\cB^*=\langle\cB^*\rangle$ and $(\cB^*)^*=\langle\cB\rangle$. Therefore, the result follows from the forward direction which we just proved.
\end{proof}

\section{Factoring weakly compact operators through a single space}\label{SectionMainTheorem}

In this section, we prove   Theorem \ref{ThmMainFactResult}. For the proof that $\cW_{\SB,\mathrm{ub}}$ and $\cW_{\SB,L_1}$ are strongly bounded with respect to the class of separable reflexive spaces admitting Schauder bases (Theorem \ref{ThmBeanlandFreeman}), we will not need any of the results in the previous sections. On the other hand, the same result for the collections   $\cW_{\mathrm{sb}^*,\SD}$ and   $\cW_{\mathrm{ub}^*,\SD}$ will make heavy use of the machinery in Section \ref{SectionDUAL} (Theorem \ref{ThmApliFactDUALS}). Since both those proofs make use of  the DFJP interpolation scheme, we start this section by recalling it (see \cite{DavisFiegelJohnsonPelczynski1974} for more details). 

Let $X$ be a Banach space and $W\subset X$ be a convex, symmetric and bounded subset. For each $n\in\N$, define a norm $|\cdot |_n$ on $X$ by letting 
\[|x|_n=\inf\Big\{\lambda>0: \frac{x}{\lambda}\in 2^nW+2^{-n}B_X\Big\}\]
for all $x\in X$ -- this is the \emph{Minkowski gauge norm on $X$ associated to $W$}. Given $p\in (1,\infty)$, define \[\Delta_p(X,W)=\Big\{x\in X:  \sum_n\|x\|^p_n<\infty\Big\}\] and let $\|\cdot \|_p$ be given by
$\|x\|_p=(\sum_n|x|_n)^{1/p}$ for all $x\in \Delta_p(X,W)$. This defines  a complete norm on $\Delta_p(X,W)$, and the space $(\Delta_p(X,W),\|\cdot\|_p)$ is called the \emph{$p$-interpolation space of the pair $(X,W)$}.

We will need  the following lemma.\footnote{The proof of Lemma \ref{PropJohnson} was given to us by W.B. Johnson  and we reproduce it here   with his permission (see \cite{MathOverFlow-Beanland}).}

\begin{lemma}\label{PropJohnson}
Let $X$ and $Y$ be Banach spaces, and $A:X \to Y$ be a weakly compact operator. Suppose $Y$ has  either an unconditional basis or is isomorphic to $L_1$, and let  $(P_n)_n$ denote either the sequence of partial sum projections associated to the unconditional basis of $Y$ or to the Haar basis of $Y$, respectively. Then
\[W=\overline{\mathrm{conv}\Big(\bigcup_{n \in \mathbb{N}} P_n(A(B_X))\Big)},\]
i.e., the closed convex hull of the union over $n\in\N$ of $P_n(A(B_X))$, is  weakly compact. 
\end{lemma}

\begin{proof}
If $Y$ has an unconditional basis the result follows straightforwardly from   \cite{Abramovic1972} (cf. \cite[Theorem 3.3]{FiJoTz-JAT}). Assume   $Y$ is  isomorphic to $L_1$, without loss of generality, we can assume that $Y=L_1$.   To see that $W$ is weakly compact we use the uniform integrability condition (see \cite[Definition 5.2.2]{AlbK-book}). Recall, a bounded subset $W'\subset L_1$ is weakly compact if and only if for all $\eps>0$ there exists $M>0$ so that $W'\subset MB_{L_2}+\eps B_{L_1}$ -- this standard fact can be deduced  from \cite[Lemma 5.2.6 and Theorem 5.2.9]{AlbK-book}. Since the latter holds for $W$, we are done. Indeed, since $A$ is weakly compact, given any $\eps>0$ there exists $M>0$ so that $A(B_X)\subset MB_{L_2}+\eps B_{L_1}$.  Since  $B_{L_2}$ and $B_{L_1}$ are invariant under the Haar basis projections, we have that $P_n(A(B_X))\subset MB_{L_2}+\eps B_{L_1}$ for all $n\in\N$, which implies   $W\subset MB_{L_2}+\eps B_{L_1}$.   \end{proof}

We now describe a construction  which will be crucial for Theorem \ref{ThmBeanlandFreeman}. Since this construction will also be used for the purposes of    Section \ref{SectionThmBorelFact}, we define the objects below in considerable more generality than we will need for Theorem \ref{ThmBeanlandFreeman}. 

Given a Polish space $S$, let $\cF(S)$ denote the space of all closed subsets of $S$ endowed with the Effros-Borel structure, i.e., the $\sigma$-algebra generated by 
\[\Big\{\{F\in \cF(S): U\cap F\neq \emptyset\}: U\subset S\text{ is open}\Big\}.\]
Given $E\in \SB$, write \[\cF_0(E)=\{F\in \cF(E): F\text{ is bounded, convex  and symmetric}\}.\]  

\begin{construction}\label{Construction} Consider  $E\in \SB$,  $y_1,y_2\in E$, and a   map $Y\in \SB(E)\mapsto (e_n^{Y})_n\in E^ \N $ so that $(e_n^{Y})_n$ is a Schauder Basis for $E$ for all $Y\in \SB(E)$. For each $Y\in \SB(E)$, let $ (P^{Y}_n)_n$ be the sequence of partial sum projections associated to $(e^{Y}_n)_n$.
 
For each $(Y,W)\in \SB(E)\times \cF_0(E)$, let
\begin{enumerate}
    \item $E_{W}=\overline{\Span}(W\cup\{y_1,y_2\})$,
    \item $W^m_{{(Y,W)}}=P^Y_m(E_{W})$,
    \item $\|\cdot\|_{{(Y,W)},m}$ be the Minkowski gauge norm on $E$ associated to $W^m_{{(Y,W)}}$,
    \item $W_{{(Y,W)}}=\overline{\mathrm{conv}(\bigcup_{m\in\N}W^m_{{(Y,W)}})}$, 
    \item $Z_{{(Y,W)}}=\Delta_2(E,W_{{(Y,W)}})$, and
    \item  $\|\cdot\|_{{(Y,W)}}$ be the norm of $\Delta_2(E,W_{{(Y,W)}})$.
\end{enumerate}
 By \cite{DavisFiegelJohnsonPelczynski1974} (or \cite[Theorem 9]{BeanlandFreeman2014}), for all $(Y,W)\in \SB(E)\times \cF_0(E)$ there are bounded operators  $B_{{(Y,W)}}:X\to Z_{{(Y,W)}}$ and $C_{{(Y,W)}}:Z_{{(Y,W)}}\to Y$ so that $A=C\circ B$ and $C_{{(Y,W)}}$ is injective. Moreover,  if the auxiliary vectors $y_1,y_2\in C(\Delta)$ are chosen so that     $e^{(Y,W)}_n\in \mathrm{span}(W_{{(Y,W)}})$ for all $n\in\N$ and all ${(Y,W)}\in \SB(E)\times \cF_0(E)$, then $e^{(Y,W)}_n\in \mathrm{Im}(C_{{(Y,W)}})$. In this case, we define
 \begin{enumerate}\setcounter{enumi}{6}
 \item $z^{{(Y,W)}}_n=C^{-1}_{{(Y,W)}}(e^{(Y,W)}_n)$ for each $n\in\N$ and ${(Y,W)}\in \SB(E)\times \cF_0(E)$.
 \end{enumerate}
Then $(z^{{(Y,W)}}_n)_n$ is a basis for $Z_{{(Y,W)}}$ for all ${(Y,W)}\in \SB(E)\times \cF_0(E)$ (see \cite{DavisFiegelJohnsonPelczynski1974} or  \cite[Theorem 9]{BeanlandFreeman2014}).
\end{construction}

The next Lemma is the version of \cite[Proposition 14]{BeanlandFreeman2014} which we need for our goals. 

\begin{lemma}\label{LemmaProposition14}
Let $E\in \SB$ be either a space with an unconditional basis or isomorphic to $L_1$, and let  $\cB\subset \cL$ be a Borel collection of weakly compact operators  all of which have $E$ as their codomain. Then there exists  reflexive $Z\in \SB$ with  a Schauder basis so that every operator in $\cB$ factors through $Z$.
\end{lemma}

\begin{proof}
This proof consists of noticing that  the methods in \cite{BeanlandFreeman2014} (mainly Proposition 14) together with  Lemma \ref{PropJohnson} give us the desired result. Precisely,   we only need to use the notation defined before the lemma so that we can use results of \cite{BeanlandFreeman2014}.

Let $(e_n)_n$ be  either the unconditional basis of $E$ or the Haar basis of $E$.  Consider the objects defined in Construction \ref{Construction} associated to $E$,  the constant assignment $(Z,W)\in \SB(E)\times \cF_0(E)\mapsto (e_n)_n\in E^\N$, and  some $y_1,y_2\in Y$ with $y_1=y_0$ and $y_2=-y_0$, where $y_0=\sum_na_ne_n\in E $ is so that $a_n\neq 0$ for all $n\in\N$. Then $y_1,y_2$ satisfy the special  condition in Construction \ref{Construction}. For each $A=(X,Y,\hat A)\in \cB$, define 
$W_A=W_{(Y,\overline{A(B_X)})}$ and $Z_A=Z_{(Y,\overline{A(B_X)})}$. By  Lemma \ref{PropJohnson},  $W_{A}$ is weakly compact for all $A\in \cB$, so $Z_{A}$ is reflexive for every such $A$. 

Since the map  $Z\in \SB(E) \mapsto (e_n)_n\in E^\N$ is Borel, the result now follows analogously to the proof of \cite[Proposition 14]{BeanlandFreeman2014}.
\end{proof}

We can now present the proof  of half of Theorem \ref{ThmMainFactResult}. As mentioned at the beginning of this section, this half only makes use of results within this current section.

\begin{theorem}\label{ThmBeanlandFreeman}
Both $\mathcal{W}_{\SB,\mathrm{ub}}$ and  $\mathcal{W}_{\SB,L_1}$ are complementably strongly bounded with respect to the class of separable reflexive spaces admitting Schauder bases. 
\end{theorem}

 \begin{proof}
The proof of this theorem is a simple adaptation of \cite[Theorem 25]{BeanlandFreeman2014} with the extra ingredient of  Lemma \ref{LemmaProposition14} and Pe\l czy\'{n}ski universal space.  

We start proving the result for $  \cW_{\SB,\mathrm{ub}}$, so let $\cA\subset \cW_{\SB,\mathrm{ub}}$ be an analytic subset. Let $U$ be Pe\l czy\'{n}ski famous universal space for unconditional basis, i.e., $U$ has an unconditional basis and every Banach space with an unconditional basis embeds into $U$ complementably \cite[Corollary 1]{Pel-Studia}. Define
\[\cA_{U}=\{(X,U,\hat A)\in \cL\colon \exists (Z,W,\hat B)\in \cA, \ X\equiv Z, \ \hat A\sim \hat B\},\]
where   $\hat A\sim \hat B$ stands for  ``the sequence $\hat A$ is equivalent to the sequence $\hat B$'' (see \cite[Section B1]{Do-Book} for definition). Since equivalence of sequences is a Borel relation in $C(\Delta)\times C(\Delta)$ and being isometric is an analytic relation in $\SB\times \SB$, $\cA_U$ is analytic. Define $\cW_U=\cW\cap (\SB\times \{U\}\times C(\Delta)^\N)$. As $\cW$ is coanalytic, so is $\cW_U$. Hence, Lusin's separation theorem gives us a Borel $\cB_U\subset \cW_U$ so that $\cA_U\subset \cB_U$. 
 
 By   Lemma \ref{LemmaProposition14}, there exists a reflexive Banach space with a Schauder basis so that every $(X,U,\hat  A)\in \cB_U$ factors through $Z$. Let us observe that every operator in $\cA$ also factors through $Z$. Indeed, let $(X,Y,\hat A)\in \cA$, and let $I_Y:Y\to U$ and $P_Y:U\to \mathrm{Im}(I_Y)$ be an isomorphic embedding and a bounded projection on $\mathrm{Im}(I_Y)$, respectively. Since $I_Y\circ A:X\to U$ is weakly compact, $I_Y\circ A:X\to U$ belongs to $\cA_U\subset \cB_U$, so there are bounded operators $B:X\to Z$ and $C:Z\to U$ so that $I_Y\circ A=C\circ B$. Hence, $A$ factors through $Z$ since $A=(I_Y^{-1}\circ P_Y\circ C)\circ B$. 
 
The result for  $  \cW_{\SB,L_1}$ is analogous, but simpler since the operators already have isomorphic codomains.
 \end{proof}

Before providing the proof for the second half of Theorem \ref{ThmMainFactResult}, we need the following lemma, which is   a simple consequence of \cite[Lemma 1]{DavisFiegelJohnsonPelczynski1974}.

\begin{lemma}\label{LemmaDFJP}
Let $X$ and $Y$ be dual  spaces, $T: Y\to X$ be a weak$^*$-to-weak$^*$ continuous weakly compact operator, and $W\subset X$ be a weakly compact, symmetric, convex, and bounded subset so that $T(B_Y)\subset W$. Let $\Delta_2(X,W)$ be the $2$-interpolation space of the pair $(X,W)$, and $J:\Delta_2(X,W)\to X$ be the standard inclusion. Then $T$ factors through $\Delta_2(X,W)$ and both $J$ and $J^{-1}\circ T$ are weak$^*$-to-weak$^*$ continuous.
\end{lemma}

\begin{proof}
Let $Z=\Delta_2(X,W)$. Since $T(B_Y)\subset W$, $J^{-1}\circ T:Y\to Z$ is well defined, so $T$ factors through $Z$, since $T=J\circ J^ {-1}\circ T$. 

Since $T$ is a weakly compact operator, $T(B_Y)$ is a relatively weakly compact subset, and \cite[Lemma 1]{DavisFiegelJohnsonPelczynski1974} implies that $Z$ is reflexive. So $J$ is weak$^*$-to-weak continuous. 

It is only left to show that $J^{-1}\circ T$ is weak$^*$-to-weak$^*$ continuous. For that, let $(y_i)_{i\in I}$ be a weak$^*$ null net in $Y$. Let $Y_*$ and $X_*$ denote the preduals of $Y$ and $X$, respectively. Since $J:Y\to X$ is weak$^*$-to-weak$^*$ continuous, there exists a bounded map $J_*:X_*\to Y_*$ such that $J=(J_*)^*$. Since $J$ is injective, it follows fro Hahn-Banach that $J_*$ has dense range. Hence, in order to show that $(J^{-1}\circ T(y_i))_{i\in I}$  is a weak$^*$ null net, it is enough to notice the following. Let $x_*\in X_*$. Then
\[
\Big(J^{-1}\circ T(y_i)\Big)\Big(J_*(x_*)\Big)=\Big((J_*)^*\circ J^{-1}\circ T(y_i) \Big)(x_*)=T(y_i) (x_*).
\]
Since $T$ is weak$^*$-to-weak$^*$ continuous, the net $(J^{-1}\circ T(y_i)(J_*(x_*)))_{i\in I}$ converges to zero. So, $J^{-1}\circ T$ is weak$^*$-to-weak$^*$ continuous.
\end{proof}

\begin{theorem}\label{ThmApliFactDUALS}
Both $\mathcal{W}_{\mathrm{ub}^*,\SD}$ and  $\mathcal{W}_{\mathrm{sb}^*,\SD}$ are complementably strongly bounded with respect to the class of separable reflexive spaces admitting Schauder bases. 
\end{theorem}

\begin{proof}
We first prove the statement for $\mathcal{W}_{\mathrm{sb}^*,\SD}$. For that, let $\cA\subset \mathcal{W}_{\mathrm{sb}^*,\SD}$ be an analytic subset. Since  $\cW$ is coanalytic  \cite[Proposition 22]{BeanlandFreeman2014},  so is $\cW_{\SD}$. Hence, Lusin's theorem \cite[Theorem 18.1]{Ke-book} gives   a Borel subset $\cB\subset \cW_{\SD}$   such that $\cA\subset \cB$. Let \[(X,Y,\hat A)\in\cB\mapsto (Y^\bullet,X^\bullet, \hat A^\bullet)\in \cL\] (or $A\in \cB\mapsto A^\bullet\in \cL$ for short) be the Borel assignment given  by Theorem \ref{ThmLSDmap}. Then  the set
\[\cA^\bullet=\{(Y^\bullet,X^\bullet,\hat A^\bullet):   (X,Y,\hat A)\in \cA\},\]
i.e., the image of $\cA$ under this assignment, 
is analytic. Since $X^\bullet$ as a shrinking basis for all $(Y^\bullet,X^\bullet,\hat A^\bullet)\in \cA^\bullet$,  \cite[Theorem 25]{BeanlandFreeman2014} gives  a  reflexive Banach space $Z$ with a Schauder basis such that every $A^\bullet\in \cA^\bullet$ factors through $Z$. Moreover, it follows from the proof of \cite[Theorem 25]{BeanlandFreeman2014} that for every $A^\bullet:Y^\bullet\to X^\bullet $ in $\cA^\bullet$ there exists a weakly compact, convex, symmetric, and bounded subset $W_A\subset X^\bullet$ such that $A^\bullet(B(Y^\bullet))\subset W_A$  and  $\Delta_2(X^\bullet,W_A)$ is isomorphic to a complemented subspace of $Z$. For each such $A^\bullet\in \cA^\bullet$, let $J_A:\Delta_2(X^\bullet,W_A)\to X^\bullet$ be the standard inclusion. 

By Theorem \ref{ThmLSDmap}, we can identify $X^*$, $Y^*$, and $A^*$ with $X^\bullet$, $Y^\bullet$, and $A^\bullet$, respectively. By Lemma \ref{LemmaDFJP}, $J_A$ and $J_A^{-1} A^\bullet$ are weak$^*$-to-weak$^ *$ continuous.  Hence, there exist maps $U_A:X\to \Delta_2(X^\bullet,W_A)_*$ and 
$V_A:\Delta_2(X^\bullet,W_A)_*\to Y$ such that $U_A^*=J_A$ and $V_A^*=J^ {-1}_A A^\bullet$, where $\Delta_2(X^\bullet,W_A)_*$ denotes   the predual of $\Delta_2(X^\bullet,W_A)$.  This gives us that 
\[(V_AU_A)^*=U_A^*V_A^*=J_AJ_A^{-1}  A^\bullet=A^\bullet=A^*.\]
So, $V_AU_A=A$. This gives us that  every $A\in \cA$ factors through the predual of $Z$, and we are done.

In order to show that   $\mathcal{W}_{\mathrm{ub}^*,\SD}$ is strongly bounded with respect to the class of separable reflexive Banach spaces with a shrinking basis, we only need to proceed exactly as for $\mathcal{W}_{\mathrm{sb}^*,\SD}$ but using Theorem \ref{ThmBeanlandFreeman} instead of \cite[Theorem 25]{BeanlandFreeman2014}. We leave this task to the reader.
\end{proof}

\begin{proof}[Proof of Theorem \ref{ThmMainFactResult}]
This follows from Theorem \ref{ThmBeanlandFreeman} and Theorem \ref{ThmApliFactDUALS}.
\end{proof}

Our penultimate result of this subsection, Theorem \ref{ThmL1},  is an application of Theorem \ref{ThmBeanlandFreeman}. But first, we need the following proposition.

\begin{proposition}\label{PropL1}
The subset $\cW_{\SB,L_1}$ is analytic. 
\end{proposition}

\begin{proof} 
 We start by showing that  $\{(X,Y,\hat A)\in \cW\colon Y=L_1\}$ is analytic. For each $n\in\N$ and each $f\in L_1$, let $f_n\in L_1$ be a function which coincides with $f$ in   $f^{-1}([-n,n])$, equals $n$ in $f^{-1}([n,\infty))$, and equal $-n$ in  $f^{-1}((-\infty,-n])$. Then the map $f\in L_1\to f_n^ 2\in L_1$ is clearly well defined and it is continuous. Since 
 \[B_{L_2}=\bigcap_{n\in\N}\Big\{f\in L_1\colon \int f_n^ 2d\mu\leq 1\Big\},\]
 it follows that $B_{L_2}$ is a closed subset of $L_1$.

In order to conclude the proof, we use once again that a bounded subset $W\subset L_1$ is weakly compact if and only if for all $\eps>0$ there exists $M>0$ so that $W\subset MB_{L_2}+\eps B_{L_1}$ (cf.\cite[Lemma 5.2.6 and Theorem 5.2.9]{AlbK-book}). Given $M,\eps>0$, define a Borel map $S_{M,\eps}:(x,y)\in L_1\times L_1\to Mx+\eps y\in L_1$. Then, letting $\cC=\{(X,Y,\hat A)\in \cW\colon Y=L_1\}$, we have that
\begin{align*}
(X,L_1,\hat A)\in \cC\ \Leftrightarrow\  &\forall \eps\in \Q_{>0},\ \exists M\in \Q_{>0},\  \forall n\in\N \\
& \hat A(n)=0\vee \frac{\hat A(n)}{\|\hat A(n)\|}\in S_{M,\eps}(B_{L_2}\times B_{L_1}).
\end{align*}
 This is an analytic condition, so   $\{(X,Y,\hat A)\in \cW\colon Y=L_1\}$ is analytic.

We will now observe that  $\cW_{\SB,L_1}$ is analytic.  For that, fix a Schauder basis $(e_j)_j$ for $L_1$ and given a basic sequence $(z_j)_j$ in $C(\Delta)$ so that $(z_j)_j\sim (e_j)_j$ let $I_{(z_j)_j}:\overline{\Span}\{z_j\colon j\in\N\}\to L_1$ be the isomorphism given by $z_j\mapsto e_j$. Then we only need to notice that 
\begin{align*}
(X,Y,\hat A)\in \cW_{\SB,L_1}\ \Leftrightarrow\ &\exists (z_j)_j\in Y^\N \big((z_j)_j\text{ is a Schauder basis for }Y\big)\\
&\wedge \big((z_j)_j\sim(e_j)_j\big)\wedge \big((X,L_1,(I_{(z_j)_j}(\hat A(n)))_n)\in \cC\big).
\end{align*}
Indeed,  the conditions  ``$(z_j)_j$ is a Schauder basis for $Y$'' and  ``$(z_j)_j\sim(e_j)_j$'' are clearly Borel conditions. Moreover, let  
\begin{align*}\cD= \{(Y,(z_j)_j,z)\in \SB & \times  C(\Delta)^\N\times C(\Delta)\colon  z\in Y\text{ and}\\
&(z_j)_j\text{ is a Schauder basis for }Y\text{ and }(z_j)_j\sim(e_j)_j\},
\end{align*}
so the assignment $(Y,(z_j)_j,z)\in \cD\mapsto I_{(z_j)_j}(z)\in L_1$ is Borel (cf. \cite[Lemma 4.8]{Braga2015Fourier}). This shows  that $\cW_{\SB,L_1}$ is analytic.
\end{proof}

\begin{proof}[Proof of Theorem \ref{ThmL1}]
This  follows from Theorem \ref{ThmBeanlandFreeman} and Proposition \ref{PropL1}.
\end{proof}

We conclude this section with the following proposition regarding the complexities of the other classes we are considering. 

\begin{proposition}
The subsets $\cW_{\SD,\mathrm{sb}}$, $\cW_{\mathrm{sb}^*,\SD}$, $\cW_{\SB,\mathrm{ub}}$, and $\cW_{\mathrm{ub}^*,\SD}$ are not analytic.
\end{proposition}

\begin{proof}
The proof consists of citing work from previous papers. First recall that the coding for the  reflexive Banach spaces with a Schauder basis, denoted by $\mathrm{REFL}_{\mathrm{b}}\subset \SB$, is complete coanalytic \cite[Corollary 3.3]{Bos-FunD}. Hence, there exists a Borel reduction $\varphi:\Tr\to\SB$ of the set of well-founded trees $\WF$ to $\mathrm{REFL}_{\mathrm{b}}$, i.e., $\varphi^{-1}(\mathrm{REFL}_{\mathrm{b}})=\WF$. Therefore, the map $\Phi:\Tr\to \cL$ given by \[\Phi(T)=(\varphi(T),\varphi(T), (d_n(\varphi(T)))_n)\in \cL\]
is a Borel reduction of $\WF$ to both $\cW_{\SD,\mathrm{sb}}$  and $\cW_{\mathrm{sb}^*,\SD}$. In particular, both $\cW_{\SD,\mathrm{sb}}$  and $\cW_{\mathrm{sb}^*,\SD}$ are not analytic. 

For the cases, $\cW_{\SB,\mathrm{ub}}$ and $\cW_{\mathrm{ub}^*,\SD}$ it suffices to consult the proof of \cite[Proposition 8.6]{BCFrWa-JFA}. Likewise, there is a Borel map $\Phi:\mathrm{Tr} \to \cL$ so that so that $\Phi^{-1}(\mathcal{W}_{\SB,\mathrm{ub}})=\mathrm{WF}$ and $\Phi^{-1}(\mathcal{W}_{\mathrm{ub}^*,\SD})=\mathrm{WF}$. 
\end{proof}

\section{Factoring weakly compact operators through subspaces of a single space}\label{SectionThmBorelFact}

The main result of this section is Theorem \ref{ThmFactoringSubspaces} below. Recall, if $S$ is a standard Borel space, then $\sigma_S(\Sigma_1^1)$ denotes the sigma algebra on $S$ generated by the analytic subsets of $S$. A map $S\to M$ between standard Borel spaces  is called \emph{$\sigma(\Sigma_1^1)$-measurable} if the preimage of every Borel subset of $M$ is in $\sigma_S(\Sigma_1^1)$.

The second named author and R. Causey proved in \cite[Theorem 5.8]{BeanlandCausey2017} that $\cW$ is strongly bounded with respect to the class of separable Banach spaces. The next result gives us that $\cW_{\SB,\SD}$ satisfies an even stronger property. Not only $\cW_{\SB,\SD}$ is strongly bounded, but the choice of the space through which the operators will factor can be done in a $\sigma(\Sigma_1^ 1)$-measurable way.

 \begin{theorem}\label{ThmFactoringSubspaces}
Let $\mathcal{B}\subset \mathcal{W}_{\SB,\SD}$ be a Borel subset. Then there is a reflexive space $Z \in \SB$ and a $\sigma(\Sigma_1^1)$-measurable $\Psi:\mathcal{B} \to \SB(Z)$ so that each $A=(X,Y,\hat A)\in \cB$ factors through $\Psi(A)$. 

Moreover, setting 
\[\cD=\{(X,Y,\hat A,x)\in \cB\times C(\Delta): x\in X\},\]
there exists a $\sigma(\Sigma_1^1)$-measurable map $\Phi:\cD\to Z$ so that, letting $\Phi_A=(X,Y,\hat A,\cdot)$, we have that, for each $A=(X,Y,\hat A)\in \cB$,
\begin{enumerate}
\item \label{Item1} $\Phi_A(X)=\Psi(A)$,
    \item\label{Item2} $\Phi_A:X\to Z$ is a bounded linear map with norm at most 1, and
    \item\label{Item3} there exists a bounded operator $L:\Psi(A)\to Y$ so that $A=L\circ \Phi_A$. 
\end{enumerate} 
\end{theorem}

The proof of Theorem \ref{ThmFactoringSubspaces} will take this entire section. We start setting some notation. Precisely, given a Borel $\B\subset \SD$, we will need to evoke a construction of a Borel map $Y\in \B\mapsto (e^Y_n)_n\in C(\Delta)^\N$ given by 
 B. Bossard. Since this construction is rather technical, we chose to simply present here the properties of it which are necessary for our settings and refer the reader to the appropriate sources.

 Precisely, B. Bossard showed that there exists a Borel map (see \cite[Claim 5.21]{Do-Book})
 \begin{equation}\label{EqBossardMap}
 Y\in \B\mapsto (e^Y_n)_n\in C(\Delta)^\N\tag{$*$}
 \end{equation}
 with the following properties:
 \begin{enumerate}
     \item $(e^Y_n)_n$ is a  monotone Schauder basis of $C(\Delta)$ for all $Y\in \B$ \cite[Page 5.13]{Do-Book}.
 \end{enumerate}
Considering the objects in Construction \ref{Construction} associated  to $C(\Delta)$, the (partial) assignment $Y\in \B\mapsto (e^{Y}_n)_n\in C(\Delta)^\N$ and $1,y\in C(\Delta)$, where $y$ is any normalized function which separates points of $\Delta$, it holds that 
\begin{enumerate}\setcounter{enumi}{1}
\item $B_Y\subset W_{(Y,B_Y)}$ for all $Y\in \B$ -- in particular, $Y$ isometrically embeds into $Z_{(Y,B_Y)}$ \cite[Lemma 4]{Kurka2016PAMS} --,
    \item $(z^{(Y,B_Y)}_n)_n$ is a monotone shrinking basis for $Z_{(Y,B_Y)}$ for all $Y\in \B$ \cite[Page 83]{Do-Book},
    \item  $Z_{(Y,B_Y)}$ is reflexive for all reflexive $Y\in \B$ \cite[Lemma 5.18]{Do-Book}, and
        \item $(Y,z)\in \B\times C(\Delta)\mapsto \|z\|_{(Y,B_Y),m}\in \R$ is Borel for all $m\in\N$ \cite[Claim 5.23]{Do-Book}.
\end{enumerate}

However, the  proofs of the properties above actually give us something stronger -- in a nutshell,  the unit ball $B_Y$ can be replaced by any closed bounded convex symmetric subset of $Y$. Precisely, the assignment \eqref{EqBossardMap} has the following stronger properties:   
\begin{enumerate} 
\item[2'.]\label{ItemContain} $W\subset W_{(Y,W)}$ for all $(Y,W)\in   \B\times \cF_0(C(\Delta))$,
    \item [3'.]\label{ItemSS} $(z^{(Y,W)}_n)_n$ is a monotone shrinking basis for $Z_{(Y,W)}$ for all $(Y,W)\in   \B\times \cF_0(C(\Delta))$,
 and
    \item [4'.]  $Z_{(Y,W)}$ is reflexive if  $W$ is weakly compact.
\end{enumerate}
Moreover, if $S$ is a standard Borel space and    $\varphi:S\to \B\times \cF_0(C(\Delta))$ is Borel, then analogous arguments as the ones in \cite[Claim 5.22 and Claim 5.23]{BeanlandFreeman2014} give us that 
\begin{enumerate}
     \item[5'.]\label{ItemBorelnorm} $(s,z)\in S\times C(\Delta)\mapsto \|z\|_{\varphi(s),m}\in \R$ is Borel for all $m\in\N$.
\end{enumerate}
The references for the stronger properties above are precisely the same as the references for the validity of their weaker versions. By 5', after $\|\cdot\|_{\varphi(s)}$-normalizing, we can replace 3' by
\begin{enumerate}
    \item [3''.]  $(z^{\varphi(s)}_n)_n$ is a normalized monotone shrinking basis for $Z_{\varphi(s)}$ for all $s\in S$.
\end{enumerate}

\subsection{Factoring a Borel $\cB\subset \cW_{\SB,\SD}$ through a family of reflexive spaces with bases.}

Let $\cB\subset \cW_{\SB,\SD}$ be a Borel subset. By Lusin's separation theorem, there exists a Borel $\B\subset \SD$ containing all the codomains of the operators in $\cB$. Let 
\[ Y\in \B\mapsto (e^Y_n)_n\in C(\Delta)^\N \]
be the map \eqref{EqBossardMap} above associated to $\B$. Moreover, consider  the objects in Construction \ref{Construction} associated  to $C(\Delta)$, the  assignment $Y\in   \B \mapsto (e^{Y}_n)_n\in C(\Delta)^\N$ and $1,y\in C(\Delta)$, where $y$ is a normalized function which separates the points of $\Delta$. In order to simplify notation, for each $A=(X,Y,\hat A)\in \cB$, let 
\[Z_A=Z_{(Y,\overline{A(B_X)})}, \ \|\cdot\|_A=\|\cdot\|_{(Y,\overline{A(B_X)})}\ \text{ and } W_A=W_{(Y,\overline{A(B_X)})}.\]
Since  the map \[(X,Y,\hat A)\in \cB\mapsto (Y,\overline{A(B_X)})\in \B\times \cF_0(C(\Delta))\]
is Borel,  we have that  $(A,z)\in \cB\times C(\Delta)\mapsto \|z\|_{A,m}\in\R$ is Borel for all $m\in\N$. Therefore,   $(A,z)\in \cB\times C(\Delta)\mapsto \|z\|_{A}$ must be Borel.

We now code the unit balls of the dual spaces $Z_A^*$ -- this will be done slightly differently than in Section \ref{SectionDUAL}. For this different coding, we fix an enumeration $(\alpha_k)_k$ of $\Q^{<\N}$ and write $\alpha_k\times (e^Y_n)$ in order to abbreviate $\sum_{j=1}^ma_je^Y_j$, where $\alpha_k=(a_1,\ldots, a_m)$. 

Given $A=(X,Y,\hat A)\in \cB$, let 
\[K_A=\Big\{w^*\in B_{\ell_\infty}:\exists f\in B_{Z_A^*},\ \forall k\in\N,\  w^*_k=\frac{f(\alpha_k\times (z^Y_n))}{\|\alpha_k\times (z^Y_n)\|_A}\Big\}.\]
So $K_A\in \cK(B_{\ell_\infty})$. Since $(\alpha_k\times (z^Y_n))_k$ is dense in $Z_A$,  $K_A$ works as a coding for the unit ball of $Z_A^*$ and it is easy to check that $K_A$ is isometric to $B_{Z_A^*}$.\footnote{This follows just as in Section \ref{SectionDUAL} for $K_{X^*}$ and $B_{X^*}$.} Define a subset $\cD\subset \cB\times [-1,1]^\N$ by letting 
\[(X,Y,\hat A,w^*)\in \cD\ \Leftrightarrow \ w^*\in K_A.\]
Since $Y\in \B\mapsto (e^Y_n)_n\in C(\Delta)^\N$ and  $(A,z)\in \cB\times C(\Delta)\mapsto \|z\|_{A}$ are Borel, it follows that $Y\in \B\mapsto\|\alpha_k\times (z^Y_n)\|_A\in \R$ is also Borel. Hence, $\cD$ is Borel.\footnote{This follows completely   analogously as the proof that  $\D$ is Borel in   \cite[Page 2430]{Braga2015Fourier}.}

\begin{lemma}\label{LemmaAssignAtoKABorel}
Let $\cB\subset \cW_{\SB,\SD}$ be Borel. The map
\[A=(X,Y,\hat A)\in \cB\mapsto K_A\in \cK(B_{\ell_\infty})\]
is Borel. Moreover, for all $A\in \cB$ there exists an isometry $i_A:K_{A}\to B_{Z_A^*}$ such that if $f=i_A(w^*)$ then $w^*_k=f(\alpha_k\times (z^Y_n))/\|\alpha_k\times (z^Y_n)\|_A$ for all $k\in\N$.
\end{lemma}

\begin{proof}
The first statement is a simple consequence of \cite[Theorem 28.8]{Ke-book} and the second statement is a trivial consequence of the interpretation of $K_{A}$ as a coding for $ B_{Z_A^*}$.
\end{proof}

Since $K_A$ is a coding for the unit ball of $Z^*_A$, we must have a Borel way to realize the duality between $Z_A$ and $K_A$. The next two lemmas take care of this task.

\begin{lemma}\label{LemmaAlpha}
Let $\cB\subset \cW_{\SB,\SD}$ be Borel and let $i_A$ be as in Lemma \ref{LemmaAssignAtoKABorel}.  Let \[\cA=\{(X,Y,\hat A,k,w^*)\in \cB\times \N\times[-1,1]^\N: w^*\in K_A\}\] and define a map $\alpha:\cA\to \R$ by 
\[\alpha(X,Y,\hat A,k,w^*)=\langle i_A(w^*),\alpha_k\times (z^Y_n)\rangle\]
for all $(X,Y,\hat A,k,w^*)\in \cA$. Then $\cA$ is a Borel set and $\alpha$ is a Borel map.
\end{lemma}

\begin{proof}
Since the assignment $A\in \cB\to K_A\in \cK([-1,1]^\N)$ is Borel by Lemma \ref{LemmaAssignAtoKABorel}, it follows that $\cA$ is Borel. Also, since for all $(X,Y,\hat A,n,w^*)\in \cA$ we have that 
\[\alpha(X,Y,\hat A,k , w^* )=w^*_n\|\alpha_k\times (z^Y_n)\|_A,\] the map $\alpha$ is clearly Borel.
\end{proof}

Given $A=(X,Y,\hat A)\in \cB$, $Z_A$ is defined as the interpolation space of the pair $(C(\Delta),W_A)$, and we can consider the standard inclusion $J_A:Z_A\to C(\Delta)$. Moreover, by 2' above, we have that  $A(B_X)\subset W_A\subset B_{Z_A}$. So,  the map $j_A:X\to Z_A$ given by $j_A(x)=J^{-1}_A(A(x))$ is well-defined and  has norm at most 1.

\begin{lemma}\label{LemmaAlphaPrime}
Let $\cB\subset \cW_{\SB,\SD}$ be Borel,  $i_A$ be as in Lemma  \ref{LemmaAssignAtoKABorel} and $j_A$ be as above.  Let \[\cA'=\{(X,Y,\hat A,x,w^*)\in \cB\times C(\Delta)\times [-1,1]^\N:x\in X,\ w^*\in [-1,1]^\N\}\] and define a map $\alpha':\cA'\to \R$ by 
\[\alpha'(X,Y,\hat A,x,w^*)=\langle i_A(w^*),j_A(x)\rangle\]
for all $(X,Y,\hat A,x,w^*)\in \cA'$. Then $\cA'$ is a Borel set and $\alpha'$ is a Borel map.
\end{lemma}

\begin{proof}
Notice that for each $(X,Y,\hat A,x,w^*)\in \A'$, 
\[\alpha'(X,Y,\hat A,x,w^*)=\lim_jw^*_{n_j}\|\alpha_{n_j}\times (z^Y_n)\|_A,\]
where $(n_j)_j$ is any sequence in  $\N$ so that $\alpha_{n_j}\times (z^Y_n) \to j_A(x)$. Therefore, 
\begin{align*}
    \{(X,Y,&\hat A,x,w^*)\in  \cA': \alpha'(X,Y,\hat A,x,w^*)\in (a,b)\}\\
    & =\bigcup_{\delta\in \Q_{>0}}\bigcap_{\eps\in \Q_{>0}}\bigcup_{k\in\N}\{(X,Y,\hat A,x,w^*)\in  \cA': \|j_A(x)-\alpha_k \times (z^Y_j)\|_A<\eps,\ \\
    &\ \ \ \ \ \ \ \ \ \ \ \ \ \ \ \ \ \ \ \ \ \ \ \ \ \   \ \ \ \ \ \ \ \ \ \ \ \ \ \ \  \ \ \ \    \alpha(X,Y,\hat A,k,w^*)\in (a+\delta,b-\delta)\}
\end{align*}
where $\alpha$ is given by Lemma \ref{LemmaAlpha}. This shows that  $\alpha'$ is Borel.
\end{proof}

We now prove the main result of this subsection.






\begin{theorem}\label{ThmFactoringBasis}
Let $\cB\subset \cW_{\SD}$ be a Borel set, and let 
\[\cE=\{(X,Y,\hat A, x)\in \cB\times C(\Delta):x\in X\}.\]
There are Borel maps 
\[\sigma:\cB\to C(\Delta)^\N\ \text{ and }\ \varphi:\cE\to C(\Delta)\]
such that, by setting $\varphi_A=(X,Y,\hat A,\cdot)$, we have that, for each $A=(X,Y,\hat A)\in \cB$,
\begin{enumerate}
    \item $\sigma(A)$ is a shrinking boundedly complete basic sequence, 
    \item $\mathrm{Im}(\varphi_A)\subset\overline{\text{span}}\{\sigma(A)\}$ and $\varphi_A:X\to \overline{\text{span}}\{\sigma(A)\}$ is a linear operator with norm at most $1$, and
    \item there exists a bounded  operator $L:\mathrm{Im}(\varphi_A)\to Y$ so that $A=L\circ \varphi_A$. Moreover, $\|L\|\leq \sup_{A\in \cB}\|J_A\|$.
\end{enumerate}
\end{theorem}

\begin{proof}
We follow the proof of \cite[Theorem 4.6]{Braga2015Fourier} closely. Let $H:\cK([-1,1]^\N)\to C(\Delta,[-1,1]^\N)$ be the map given in Lemma \ref{LemmaH}, and $\alpha$ and $\alpha'$ be the maps in Lemma \ref{LemmaAlpha} and Lemma \ref{LemmaAlphaPrime}, respectively. Fix a sequence $(n_k)_k$ in $\N^\N$ so that $\alpha_{n_k}\times (e^Y_j)=e^Y_k$ for all $k\in\N$ -- notice that this sequence does not depend on $Y$. For each $A=(X,Y,\hat A)\in \cB$, we define 
\[\sigma(A)=\Big(\alpha\big(X,Y,\hat A, n_k, H(K_A)(\cdot)\big)\Big)_k.\]
Since $A\in \cB\mapsto K_A\in \cK([-1,1]^\N)$ is Borel (Lemma \ref{LemmaAssignAtoKABorel}), $\sigma$ is clearly Borel.

\begin{claim}
The sequence $\sigma(A)$ is $1$-equivalent to $(z^Y_n)_n$ for all  $A\in \cB$. In particular $\sigma(A)$ is a shrinking  boundedly complete basic sequence.
\end{claim}

\begin{proof}
One only needs to notice that the assignment
\[\alpha_k\times (z^Y_n)\in Z_A\mapsto \alpha\big(X,Y,\hat A,k,H(K_A)(\cdot)\big)\in \overline{\Span}\{\sigma(A)\}\]
defines a surjective linear  isometry $Z_A\to \overline{\Span}\{\sigma(A)\}$. Since this follows exactly as in the proof of Theorem 4.6, we leave the details to the reader. 
\end{proof}
 
Define $\varphi:\cE\to C(\Delta)$ by letting
\[\varphi(X,Y,\hat A, x)=\alpha'\big(X,Y,\hat A, x, H(K_A)(\cdot)\big),\]
for all $A=(X,Y,\hat A,x)\in \cE$. By Lemma \ref{LemmaAlphaPrime}, $\varphi $ is Borel. Given $A=(X,Y,\hat A)\in \cB$, the map $\varphi_A=\varphi(X,Y,\hat A, \cdot)$ is the composition of the map $j_A:X\to Z_A$ in Lemma \ref{LemmaAlphaPrime} with the isometry $I_A:Z_A\to \overline{\Span}\{\sigma(A)\}$, so $\mathrm{Im}(\varphi_A)\subset \overline{\Span}\{\sigma(A)\}$. Moreover, since $\|j_A\|\leq 1$, the linear operator $\varphi_A$ has norm at most $1$. At last, since $\varphi_A=I_A\circ j_A$, we have that $A=(J_A\circ I^{-1}_A)\circ \varphi_A$, so we are done.
\end{proof}

\begin{corollary}
Let $\cB\subset \cW_{\SB,\SD}$ be Borel. There exists a Borel map $\Psi:\cB\to \SB$ so that for all $A\in \cB$
\begin{enumerate}
\item $\Psi(A)$ is a reflexive space, and
\item $A$ factors through $\Psi(A)$.
\end{enumerate}
\end{corollary}

\begin{proof}
Let $\cE$ and $\varphi:\cB\to C(\Delta)^\N$ be given by Theorem \ref{ThmFactoringBasis} and define $\Psi:\cB\to \SB$ by letting $\Psi(A)=\overline{\{\varphi(X,Y,\hat A,d_n(X)):n\in\N\}}$ for all $A=(X,Y,\hat A)\in \cB$. Since $\varphi$ is Borel, so is $\Psi$, and the other properties follows from Theorem \ref{ThmFactoringBasis}.  
\end{proof}
\subsection{Coding by rational spaces and amalgamation}\label{SubsecRationalAmalg}

We have seen in the previous subsection how to construct a Borel assignment $\sigma:\cB\to C(\Delta)^\N$ so that every $A\in \cB$ factors through a subspace of $\sigma(A)$. In this subsection, we code each $\sigma(A)$ as a subspace of a rational Banach space, and use the amalgamation method presented in \cite{Kurka-StudiaAmal} in order to construct a single reflexive space containing every $\sigma(A)$.

\begin{definition}
Let $d\in\N$.  A norm $\|\cdot\|_d$ on $\R$ is called \emph{rational} if its unit ball is the convex hull of  finitely many points whose coordinates in the standard basis of $\R^d$ are all rational. 
\end{definition}

Given $\bar n= (n_j)_j\in \N^{<\N}$, let us define a norm $\|\cdot\|_{\bar n}$ on $\R^{|\bar n|}$, where $|\bar n|$ denotes the length of the tuple $\bar n$. For that, for each $d\in\N$, fix an enumeration $(\|\cdot\|_{d,j})_j$  of all monotone rational norms on $\R^d$. Then, given $\bar n=(n_j)_j\in \N^{<\N}$, define a norm $\|\cdot\|_{\bar n}$ on $\R^{|\bar n|}$ by letting the unit ball of $(\R^{|\bar n|},\|\cdot\|_{\bar n})$ be
\[B_{(\R^{\bar n},\|\cdot\|_{\bar n})}=\overline{\mathrm{conv}}\Big(\bigcup_{j=1}^{|\bar n|}B_{(\R^j, \|\cdot\|_{j,n_j})}\Big).\]
If $\bar n=(n_j)_j\in \N^\N$, we define a norm $\|\cdot\|_{\bar n}$ on a subspace of $\R^\N$ analogously. To simplify notation, we denote the Banach spaces just defined by $(F_{\bar n},\|\cdot\|_{\bar n})$  regardless of  $\bar n$ being a finite tuple or not. For each such $\bar n$, we denote the standard basis of $F_{\bar n}$ by $(r_n)_n$.

Let \[\mathrm{mbs}=\{\bar f\in C(\Delta)^\N: \bar f  \text{ is a normalized monotone basic sequence}\},\] so  $\mathrm{mbs}$ is a Borel subset of $C(\Delta)^\N$. We will now define a map $\psi:\mathrm{mbs}\to \N^\N$ which was implicitly defined in  \cite[Section 6]{Kurka-StudiaAmal}.\footnote{This is the important step noticed by O. Kurka in \cite{Kurka-StudiaAmal} which allows us to obtain isometric statements.}

Fix a bijection $\pi=(\pi_1,\pi_2):\N\to \N\times \N$   so that both $\pi_1(m)\leq \pi_1(k)$ and $\pi_2(m)\leq \pi_2(k)$ imply $m\leq k$ (e.g., let $\pi$ be the bijection in \cite[Definition 6.3]{Kurka-StudiaAmal}). For each $\bar f=(f_n)_n\in \mathrm{mbs}$, let $(f_{j,n})_{j,n}$ denote the basis of   $\ell_2(\overline{\Span}\{\bar f\})$ so that $f_{j,n}$ belongs to the $j$-th copy of $\overline{\Span}\{\bar f\}$ in $ \ell_2(\overline{\Span}\{\bar f\})$ and equals $f_n$. For each $i\in\N$, let $g_i=f_{\pi(i)}$, so $(g_i)_i$ is a monotone basis for $\ell_2(\overline{\Span}\{\bar f\})$, for all $\bar f\in \mathrm{mbs}$.

We define  the map $\psi:\mathrm{mbs}\to \N^\N$ as follows. Given $\bar f\in \mathrm{mbs}$, and  $d\in \N$, let $n_d\in \N$ be the least natural number so that 
\begin{align*}
\Big(1-\frac{1}{2d+1}\Big)\Big\|\sum_{i=1}^d a_ig_i\Big\|_{\ell_2(\overline{\Span}\{\bar f\})}&\leq \Big\|\sum_{i=1}^da_ir_i\Big\|_{d,n_d}\\
&\leq \Big(1-\frac{1}{2d+2}\Big)\Big\|\sum_{i=1}^da_ig_i\Big\|_{\ell_2(\overline{\Span}\{\bar f\})}
\end{align*}
for all $(a_i)_i\in \R^d$. Such $n_d$ exists since $(g_i)_i$ is a  monotone basis of $\ell_2(\overline{\Span}\{\bar f\})$. Set $\psi(\bar f)=(n_d)_d$, this map is clearly Borel.

The assignment $g_i\mapsto r_i$ defines an isomorphism between   $\ell_2(\overline{\Span}\{\bar f\})$ and $F_{\psi(\bar f)}$ for all $\bar f\in \mathrm{mbs}$. For each $j\in\N$ and $\bar f\in \mathrm{mbs}$, we denote the canonical inclusion of $\overline{\Span}\{\bar f\}$ onto its $j$-th  copy in $F_{\psi(\bar f)}$ by $U_{\bar f,m}$. Precisely,  $U_{\bar f,m}:\overline{\Span}\{\bar f\}\to F_{\psi(\bar f)}$ is the   linear map given by 
\[U_{\bar f,j}(f_n)=r_{\pi^{-1}(j, n)}\]
for all $n\in\N$.

\begin{lemma}\label{LemmaRational}
Let  $\psi:\mathrm{mbs}\to \N^\N$ be the Borel map defined above. Then, for all $\bar f\in \mathrm{mbs}$,
\begin{enumerate}
    \item $F_{\psi(\bar f)}$ is isomorphic to $\ell_2(\overline{\Span}\{\bar f\})$,
    \item if $\bar f$ is shrinking, then $(r_n)_n$ is a shrinking basis for $F_{\psi(\bar f)}$, and
    \item the map $U_{\bar f}:\overline{\Span}\{\bar f\}\to F_{\psi(\bar f)}$ given by 
    \[U_{\bar f}(x)=\frac{\sqrt{3}}{2}\cdot\sum_{j\in\N}\frac{1}{2^{j-1}}U_{\bar f,j}(x)\]
    is an isometry and its image is a $1$-complemented subspace of 
    $F_{\psi(\bar n)}$.
\end{enumerate}
\end{lemma}

\begin{proof}
The map $\psi$ is clearly Borel, and the other statements are precisely the content of \cite[Proposition 6.2 and Lemma 6.7]{Kurka-StudiaAmal}
\end{proof}

\begin{proof}[Proof of Theorem \ref{ThmFactoringSubspaces}]
Let $\cB\subset \cW_{\SB,\SD}$ be a Borel subset, and let $\cE\subset \cB\times C(\Delta)$, $\sigma:\cB\to C(\Delta)^\N$ and $\varphi:\cE\to C(\Delta)$ be as in Theorem \ref{ThmFactoringBasis}. Without loss of generality, we can assume that $\sigma(A)$ is normalized, for each $A\in\cB$. Let $\psi:\mathrm{mbs}\to \N^\N$ be given by Lemma \ref{LemmaRational}. Then  $\A=\psi(\sigma(\cB))$ is an analytic subset of the Baire space $\N^\N$. Hence, there exists a pruned tree $T$ on $\N\times \N$ so that $\A=p([T])$, where  $p:\N^\N\times \N^\N\to \N^\N$ denotes the projection onto the first coordinate (see \cite[Theorem 25.2]{Ke-book}).

Define a norm $\|\cdot\|$ on $c_{00}(T)$ by letting
\[\|x\|=\sup_{\beta\in [T]}\Big\|\sum_{t\prec  \beta}x(t)r_{|t|}\Big\|_{p(\beta)}\]
for each $x=(x(t))_{t\in T}\in c_{00}(T)$. Let $E$ be the completion of $c_{00}(T)$ under the norm $\|\cdot\|$, and let $(e_t)_{t\in T}$ denote the canonical basis of $c_{00}(T)$, i.e., $e_t(t)=1$ and $e_t(s)=0$ for all $s\neq t$. So $(e_t)_{t\in T}$ is a basis for $E$ (see \cite[Definition 3.1]{Kurka-StudiaAmal} for details).

Clearly, if $\bar n=\psi(\sigma(A))$ for some $A\in \cB$, then $F_{\bar n}$ is isometric to a subspace of $E$. Indeed, let $\bar m\in \N^\N$ be so that $\beta=(\bar n,\bar m)\in [T]$. Then the assignment $r_i\mapsto e_{\beta\restriction {i}}$ defines an isometry between  $F_{\bar n}$ and \[E_{\beta}=\{x=(x(t))_{t\in T}\in E: t\not\prec \beta\Rightarrow x(t)=0\}.\]
However, we need a $\sigma(\Sigma_1^1)$-measurable  way of choosing such $\bar m\in\N$. For that, notice that $[T]$ is a closed subset of $\N^\N\times\N^\N$, so  Jankov-von Neumann uniformization theorem gives a $\sigma(\Sigma_1^1)$-measurable map $\theta:\A\to [T]$ which uniformizes $[T]$ (see \cite[Theorem 18.1]{Ke-book}), i.e., $p(\theta(\beta))=\beta$ for all $\beta\in \A$. For each $A\in \cB$, define $t_{A}=\theta(\psi(\sigma(A)))$. So  $F_{\psi(\sigma(A))}$ is canonically isomorphic to $E_{t_A}$ for all $A\in \cB$.

Define $W=\overline{\mathrm{conv}}\bigcup_{\beta\in [T]}B_{E_{\beta}}$, and let $Z=\Delta_2(E,W)$. Since $Z$ is separable, by fixing an isometric copy of $Z$ in $\SB$, we can assume without loss of generality  that $Z\in \SB$.

\begin{claim}
The interpolation space $Z$ is reflexive.
\end{claim}
\begin{proof}
By Theorem \ref{ThmFactoringBasis}, it follows that $\sigma(A)$ is a boundedly complete shrinking basis for all $A\in \cB$, so $\overline{\Span}\{\sigma(A)\}$ is reflexive. Hence, by Lemma \ref{LemmaRational}, $F_{\psi(\sigma(A))}$ is reflexive for all $A\in \cB$. Then $E_{\beta}$ is reflexive for all $\beta\in [T]$ and   \cite[Proposition 4.6]{Kurka-StudiaAmal} gives us that $Z$ is reflexive.
\end{proof}

 Since $B_{E_\beta}\subset W$ for all $\beta\in [T]$, the natural inclusion  $i_\beta:E_\beta\to Z$ is well defined and bounded for all $\beta\in [T]$. Moreover, by \cite[Lemma 4.2 and Fact 4.4]{Kurka-StudiaAmal}, there exists $c>0$ so that the map $x\in E_{\beta}\mapsto ci_\beta(x)\in Z$   is an isometry  for all $\beta\in [T]$.  For each $A\in \cB$ and $k\in\N$, let $t_{A,k}\in T$ be the initial segment of $t_A$ with length $k$, and define a map $I_A:F_{\psi(\sigma(A))}\to Z$ by
\[I_A:\sum_{k\in\N}a_kr_k\in F_{\psi(\sigma(A))}\mapsto ci_{t_A}\Big(\sum_{k\in\N}a_ke_{t_{A,k}}\Big)\in Z.\]
Then $I_A$ is an isometric embedding for all $A\in \cB$.

Let $\cD$ be as in the statement of Theorem \ref{ThmFactoringSubspaces}. We define the desired maps $\Phi:\cD\to Z$ and $\Psi:\cB\to \SB(Z)$ by letting
\[\Phi(X,Y,\hat A,x)=I_A\Big(U_{\sigma(A)}\big(\varphi(X,Y,\hat A,x)\big)\Big)\]
for all $(X,Y,\hat A,x)\in \cD$ and
\[\Psi(X,Y,\hat A)=\overline{\Span}\Big\{\Phi(X,Y,\hat A, d_n(X)): n\in\N\Big\}\]
for all $(X,Y,\hat A)\in \cB$, where $(d_n)_n$ are the Kuratowski
and Ryll-Nardzewski's Borel selectors (see   Subsection \ref{SubsectionCodeBoundedOp}).

Since $\varphi_A=\varphi(X,Y,\hat A,\cdot):X\to \overline{\Span}\{\sigma(A)\}$ is a bounded linear map with norm at most 1, and both $I_A$ and $U_{\sigma(A)}$ are isometries,  it follows that $\Phi_A=\Phi(X,Y,\hat A,\cdot): X\to Z$ is a bounded linear map with norm at most 1 for all $A\in \cB$. So \eqref{Item2} holds. Moreover,  this gives that $\Psi(X,Y,\hat A)\in \SB(Z)$, so $\Psi$ is well defined. 

Given $A=(X,Y,\hat A)\in \cB$,   Theorem \ref{ThmFactoringBasis} gives   a bounded operator $L:\mathrm{Im}(\varphi_A)\to Y$ so that $A=L\circ \varphi_A$ and $\|L\|\leq\sup_{A\in \cB}\|J_A\|$. Hence, as $\Phi_A=I_A\circ U_{\sigma(A)}\circ \varphi_A$, it follows that \[A=\big(L\circ U_{\sigma(A)}^{-1}\circ I_A^{-1}\big)\circ\Phi_A.\] So $A$ factors through $\Psi(X,Y,\hat A)$ for all $(X,Y,\hat A)\in \cB$ and \eqref{Item3} holds.

\begin{claim}
The maps $\Phi$ and $\Psi$ are $\sigma(\Sigma_1^1)$-measurable.
\end{claim}

\begin{proof}
By the definition of $\Psi$, it is enough to show that $\Phi$ is $\sigma(\Sigma_1^1)$-measurable. Firstly, notice that, as $A\in \cB\to t_A\in \N^\N\times \N^\N$ is $\sigma(\Sigma_1^1)$-measurable, the assignment \[A\in \cB\mapsto (i_{t_A}(e_{t_{A,k}}))_k\in C(\Delta)^\N\] is also $\sigma(\Sigma_1^1)$-measurable.

Let $U\subset C(\Delta)$ be an open ball. To simplify notation, for each $A=(X,Y,\hat A)\in \cB$ and $j\in\N$, let $\sigma(A)=(\sigma_{A,i})_i$. Notice that, by the definition of $\Phi_A$, we have that
\[\Phi_A\Big(\sum_{i=1}^ka_i\sigma_{A,i}\Big)=c\cdot\frac{\sqrt{3}}{2}\sum_{j\in\N}\frac{1}{2^{j-1}}\sum_{i=1}^ka_ii_{t_A}\Big(e_{t_{A,\pi^{-1}(j,i)}}\Big)\]
for all $k\in\N$ and all $a_1,\ldots, a_k\in \R$. Then
\begin{align*}
\Phi_A(x)\in U\ \Leftrightarrow \ &\exists \delta\in \Q_{>0},\ \exists k\in\N,\ \exists a_1,\ldots a_k\in \Q \\
& \Big(\Big\|\varphi_A(x)-\sum_{i=1}^ka_i\sigma_{A,i}\Big\|<\delta\Big)\wedge \Big(\forall m\in\N\\ &\Big\|d_m(C(\Delta))-c\cdot\frac{\sqrt{3}}{2}\sum_{j\in\N}\frac{1}{2^{j-1}}\sum_{i=1}^ka_ii_{t_A}\Big(e_{t_{A,\pi^{-1}(j,i)}}\Big)
\Big\|<\delta\\
&\to d_m(C(\Delta))\in U\Big).
\end{align*}
This shows that $\Phi$ is $\sigma(\Sigma_1^1)$-measurable.
\end{proof}

Since we clearly have that $\Phi_A(X)=\Psi(X)$ for all $(X,Y,\hat A)\in \cB$,  \eqref{Item1} holds and the proof is  finished.
\end{proof}

\begin{proof}[Proof of Theorem \ref{ThmFactoringSubspacesInto}]
This is simply the first statement of Theorem \ref{ThmFactoringSubspaces}.
\end{proof}

P. Dodos and V. Ferenczi showed that $\mathrm{REFL}=\{X\in \SB: X\text{ is reflexive}\}$ is a strongly bounded class of Banach spaces. Theorem \ref{ThmFactoringSubspaces} allow us to strengthen this result. Precisely, we have the following corollary.

\begin{corollary}
 \label{CorSBRefl}
Say $\mathbb{B}\subset \mathrm{REFL}$ is Borel. There exists a $Z\in\mathrm{REFL}$ with a basis, and a $\sigma(\Sigma_1^1)$-measurable map $\Psi:\mathbb{B}\to \SB(Z)$ such that $X\equiv \Psi(X)$, for all $X\in \mathbb{B}$. Moreover, setting \[\mathbb{E}=\{(X,x)\in \mathbb{B}\times C(\Delta)\ |\ x\in X\},\] there exists a $\sigma(\Sigma_1^1)$-measurable map
\[\psi: \mathbb{E}\to Z\]
such that, letting $\psi_X=\psi(X,\cdot)$, we have that $\psi_X:X\to Z$  is an isometric embedding for all $X\in\mathbb{B}$. 
\end{corollary}

\begin{proof}
Let $\cB=\{(X,X,(d_n(X))_n)\in \cL: X\in \B\}$. Notice that $(d_n(X))_n$ codes the identity operator on $X$. Since $\B\subset \mathrm{REFL}$, we have that $\cB\subset \cW_{\SD}$. Let $Z\in \mathrm{REFL}$, $\Psi:\cB\to \SB(Z)$, $\cD$, and $\Phi:\cD\to Z$ be given by Theorem \ref{ThmFactoringSubspaces}.

It was shown in the proof of Theorem \ref{ThmFactoringSubspaces} that for all $X\in \B$, there exists a bounded operator $L:\Psi(X)\to Y$ so that $\mathrm{Id}_X=L\circ \Phi_{\mathrm{Id}_X}$ and $\|L\|\leq \sup_{X\in \B}\|J_{\mathrm{Id}_X}\|$, where $J_{\mathrm{Id}_X}:X\to Z_{\mathrm{Id}_X}$ is as in Theorem \ref{ThmFactoringBasis}. Hence, since we are dealing with the identity operators, we have that $B_X\subset W_{\mathrm{Id}_X}$ for all $X\in \B$. In particular, $\|J_{\mathrm{Id}_X}\|\leq 1$ for all $X\in \B$, so $\|L\|\leq 1$. Since $\|\Phi_{\mathrm{Id}_X}\|\leq 1$, we conclude that $\Phi_{\mathrm{Id}_X}:X\to \Psi(X)$ is an isometric embedding. 
\end{proof}

We finish this section with a corollary and a remark about it and its proof -- which we choose to omit. 

\begin{corollary}\label{CorSBShrink}
Say $\mathbb{B}\subset \SD$ is Borel. There exists a $Z\in \SD$, with a shrinking basis, and a $\sigma(\Sigma_1^1)$-measurable map $\Psi:\mathbb{B}\to \SB(Z)$ such that $X\equiv \Psi(X)$, for all $X\in \mathbb{B}$. Moreover, setting \[\mathbb{E}=\{(X,x)\in \mathbb{B}\times C(\Delta)\ |\ x\in X\},\] there exists a $\sigma(\Sigma_1^1)$-measurable map
\[\psi: \mathbb{E}\to Z\]
such that, letting $\psi_X=\psi(X,\cdot)$, we have that $\psi_X:X\to Z$  is an isometric embedding  for all $X\in\mathbb{B}$. \qed
\end{corollary}
 
\begin{remark}
The proof of Corollary \ref{CorSBShrink} is a  is a mix  between the proofs of \cite[Theorem 1.5]{Braga2015Fourier} and Corollary \ref{CorSBRefl} above. Since there are no technical difficulties in it,  we leave the details to the reader. We point out that an isomorphic version of Corollary \ref{CorSBShrink} in which the maps obtained are claimed to be Borel measurable was published by the third named author in \cite[Theorem 1.5]{Braga2015Fourier}. However, its proof is incorrect, and the maps obtained in \cite[Theorem 1.5]{Braga2015Fourier} are actually only $\sigma(\Sigma_1^1)$-measurable -- the mistake is contained in \cite[Lemma 4.7]{Braga2015Fourier}. Hence, Corollary \ref{CorSBShrink} is a strengthening of (the corrected version of) \cite[Theorem 1.5]{Braga2015Fourier}.
\end{remark}

\section{Concluding remarks and questions}\label{SectionOpenProb}

We finish this paper with some natural questions which are left open. Firstly, it would be very interesting to get rid of $\sigma(\Sigma_1^1)$-measurability in the theorems in Subsection \ref{SubsecRationalAmalg}. Since an arbitrary analytic subset of $\A\subset\N^\N\times \N^\N$ may not have a Borel uniformization even if $p(\A)=\N^\N$, a proof for that will probably require some new ideas and a different amalgamation method.

\begin{question}
Do the Borel measurable versions of Theorem \ref{ThmFactoringSubspaces}, Corollary \ref{CorSBRefl} and Corollary \ref{CorSBShrink} hold?
\end{question}




As noticed in the introduction, if $\cC$ is a collection of weakly compact operators containing the identity on a Banach space without the bounded approximation property, then there is no hope of finding a reflexive space $Z$ with a Schauder basis so that all members of $\cC$ factor through $A$. Although Theorem \ref{ThmMainFactResult} gives us some conditions on $\cC$ for which a positive answer holds, we are very far from completely understanding what are the precise conditions  for such $Z$ to exist.

\begin{question}
Let $\cC\subset \cW$ be a set of weakly compact operators so that each element of $\cC$ factors through a reflexive Banach space with a Schauder basis. Is $\cC$ strongly bounded?   
\end{question}



Recall that if a Banach space $Y$ has a basis so that $Y^*$ is separable and has the bounded approximation property, then $Y$ has a shrinking basis \cite{JohnsonRosenthalZippin1971Israel}. Hence, Theorem \ref{BF-main} applies to a collection $\cC$ whose domain spaces satisfies this property. However, if $Y^*$ is not assumed to be separable, the question remains open. Precisely, the following was asked by W.B. Johnson in MathOverflow \cite{MathOverFlow-Beanland}.

\begin{question} 
Let $A:X \to Y$ be weakly compact and suppose $Y$ has a basis and $Y^*$ has the bounded approximation property. Is there a reflexive $Z$ with a basis so that $A$ factors though $Z$?
\end{question}

Our final problem is related to an open problem of P. Dodos from \cite{Do-Houston}. For  the definition of coanalytic ranks, see \cite[Appendix A]{Do-Book}.

\begin{question}
Let $\phi$ be a coanalytic rank on $\mathcal{L}_{\SD}$ and let $\mathcal{A}$ be an analytic subset of dual operators in $\mathcal{L}$. Suppose that for each $A \in \mathcal{A}$ there is a countable ordinal $\xi_A$ so that 
$$\sup \{ \phi(B) : B^*\mbox{ is isomorphic to }  A\}<\xi_A.$$
Is the set 
$$\mathcal{A}_* =\{B \in \mathcal{L}_{\SD} :\exists A \in \mathcal{A} \mbox{ where } B^* \mbox{ is isomorphic to } A\}$$
analytic? 
\end{question}

\begin{acknowledgments}
Part of this paper was done while the third named author was visiting the Universidade de S\~{a}o Paulo (USP), Brazil, in May of 2019. This author would like to thank Valentin Ferenczi and Christina Brech for their great hospitality.
\end{acknowledgments}

\bibliographystyle{abbrv}
\bibliography{bib_beanland}

\end{document}